\documentclass[11pt,reqno]{amsart}

\usepackage{amssymb,amscd,amsxtra,soul,enumerate,hyperref,esint}
\hypersetup{colorlinks=false}

\numberwithin{equation}{section}

\theoremstyle{plain}

\newtheorem{thm}{Theorem}[section]
\newtheorem{lem}[thm]{Lemma}
\newtheorem{cor}[thm]{Corollary}
\newtheorem{prop}[thm]{Proposition}

\theoremstyle{definition}

\newtheorem{defn}[thm]{Definition}

\newtheorem{ex}[thm]{Example}

\theoremstyle{remark}

\newtheorem{rem}[thm]{Remark}

\newcommand{\C}{\mathbb{C}}

\newcommand{\N}{\mathbb{N}}

\newcommand{\R}{\mathbb{R}}

\newcommand{\Z}{\mathbb{Z}}

\renewcommand{\AA}{\mathcal{A}}

\newcommand{\FF}{\mathcal{F}}

\newcommand{\HH}{\mathcal{H}}

\newcommand{\LL}{\mathcal{L}}

\newcommand{\RR}{\mathcal{R}}
\renewcommand{\SS}{\mathcal{S}}

\newcommand{\fg}{\mathfrak{g}}

\newcommand{\fG}{\mathfrak{G}}

\newcommand{\fU}{\mathfrak{U}}
\newcommand{\fX}{\mathfrak{X}}

\newcommand{\bfh}{\mathbf{h}}

\newcommand{\bfr}{\mathbf{r}}
\newcommand{\bfs}{\mathbf{s}}

\newcommand{\bfH}{\mathbf{H}}

\newcommand{\bfV}{\mathbf{V}}

\newcommand{\sA}{\mathsf{A}}

\newcommand{\sT}{\mathsf{T}}

\newcommand{\supp}{\operatorname{supp}}
\newcommand{\codim}{\operatorname{codim}}
\newcommand{\im}{\operatorname{im}}

\newcommand{\Cl}{\operatorname{Cl}}
\newcommand{\id}{\operatorname{id}}
\newcommand{\Tr}{\operatorname{Tr}}

\newcommand{\Aut}{\operatorname{Aut}}

\newcommand{\Diffeo}{\operatorname{Diffeo}}
\newcommand{\Diff}{\operatorname{Diff}}
\newcommand{\GL}{\operatorname{GL}}
\newcommand{\Hol}{\operatorname{Hol}}
\newcommand{\Hom}{\operatorname{Hom}}

\newcommand{\rank}{\operatorname{rank}}
\newcommand{\End}{\operatorname{End}}
\newcommand{\dom}{\operatorname{dom}}

\newcommand{\ev}{\operatorname{ev}}
\newcommand{\Pen}{\operatorname{Pen}}

\newcommand{\inj}{\operatorname{inj}}
\newcommand{\pr}{\operatorname{pr}}

\newcommand{\length}{\operatorname{length}}

\newcommand{\rnabla}{\mathring{\nabla}}
\newcommand{\rP}{\mathring{P}}

\newcommand{\rexp}{\mathring{\exp}}
\newcommand{\rGamma}{\mathring{\Gamma}}
\newcommand{\cnabla}{\check{\nabla}}
\newcommand{\cexp}{\check{\exp}}
\newcommand{\cO}{\check{O}}
\newcommand{\cV}{\check{V}}
\newcommand{\Trs}{\Tr^{\text{\rm s}}}

\newcommand{\Fsigma}{{}^\FF\!\!\sigma}
\newcommand{\Diffub}{\Diff_{\text{\rm ub}}}
\newcommand{\Cinftyub}{C^\infty_{\text{\rm ub}}}
\newcommand{\Cinftyc}{C^\infty_{\text{\rm c}}}

\newcommand{\sm}{\smallsetminus}

\newcommand{\Ldis}{L_{\text{\rm dis}}}
\newcommand{\LdisK}{L_{\text{\rm dis},K}}
\newcommand{\LdisJ}{L_{\text{\rm dis},J}}
\newcommand{\olfX}{\overline{\fX}}
\newcommand{\fXcom}{\fX_{\text{\rm com}}}
\newcommand{\olfXcom}{\overline{\fX}_{\text{\rm com}}}

\usepackage{color}
\definecolor{darkgreen}{cmyk}{1,0,1,.2}
\definecolor{m}{rgb}{1,0.1,1}

\newdimen\theight
\def\TeXref#1{%
             \leavevmode\vadjust{\setbox0=\hbox{{\tt
                     \quad\quad  {\small \textrm #1}}}%
             \theight=\ht0
             \advance\theight by \lineskip
             \kern -\theight \vbox to
             \theight{\rightline{\rlap{\box0}}%
             \vss}%
             }}%

\title{Analysis on Riemannian foliations of bounded geometry}

\author[J.A. \'Alvarez L\'opez]{Jes\'us A. \'Alvarez L\'opez}
\address{Department/Institute of Mathematics\\
         University of Santiago de Compostela\\
         15782 Santiago de Compostela\\ Spain}
\email{jesus.alvarez@usc.es}

\author[Y.A. Kordyukov]{Yuri A. Kordyukov}
\address{Institute of Mathematics\\ 
Ufa Federal Research Centre\\ 
Russian Academy of Science\\ 
112 Chernyshevsky str.\\ 
450008 Ufa\\ Russia}
\email{yurikor@matem.anrb.ru}

\thanks{The authors are partially supported by FEDER/Ministerio de Ciencia, Innovaci\'on y Universidades/AEI/MTM2017-89686-P and MTM2014-56950-P, and Xunta de Galicia/2015 GPC GI-1574. The second author is partially supported by the RFBR grant 16-01-00312.}

\author[E. Leichtnam]{Eric Leichtnam}
\address{Institut de Math\'ematiques de Jussieu-PRG\\ CNRS\\ Batiment Sophie Germain (bureau 740)\\ Case~7012\\ 75205 Paris Cedex 13, France}
\email{ericleichtnam@math.jussieu.fr}

\dedicatory{Tribute to Christopher Deninger for his 60th birthday}

\date{\today}



\begin{document}

\maketitle

\begin{abstract}
A leafwise Hodge decomposition was proved by Sanguiao for Riemannian foliations of bounded geometry. Its proof is explained again in terms of our study of bounded geometry for Riemannian foliations. It is used to associate smoothing operators to foliated flows, and describe their Schwartz kernels. All of this is extended to a leafwise version of the Novikov differential complex.
\end{abstract}

\tableofcontents

\section{Introduction}

Christopher Deninger has proposed a program to study arithmetic zeta functions by finding an interpretation of the so called explicit formulae as a (dynamical) Lefschetz trace formula for foliated flows on suitable foliated spaces \cite{Deninger1998,Deninger2001,Deninger2002,Deninger-Arith_geom_anal_fol_sps,Deninger2008}. Hypothetically, the action of the flow on some reduced leafwise cohomology should have some Lefschetz distribution. Then the trace formula would describe it using local data from the fixed points and closed orbits. The precise expression of these contributions was previously suggested by Guillemin \cite{Guillemin1977}. Further developments of these ideas were made in \cite{DeningerSinghof2002,Leichtnam2008,Leichtnam2014,Kim-fiber_bdls}.

Deninger's program needs the existence of foliated spaces of arithmetic nature, where the application of the trace formula has arithmetic consequences. Perhaps some generalization of foliated spaces should be considered. Anyway, to begin with, we consider a simple foliated flow $\phi=\{\phi^t\}$ on a smooth closed foliated manifold $(M,\FF)$. We assume that $\FF$ is of codimension one and the orbits of $\phi$ are transverse to the leaves without fixed points.

The first two authors proved such a trace formula when $\{\phi^t\}$ has no fixed points \cite{AlvKordy2002}. A generalization for transverse actions of Lie groups was also given \cite{AlvKordy2008a}. It uses the space $C^\infty(M;\Lambda\FF)$ of leafwise forms (smooth sections of $\Lambda\FF=\bigwedge T\FF^*$ over $M$), which is a differential complex with the leafwise derivative $d_\FF$. Its reduced cohomology is denoted by $\bar H^*(\FF)$ (the leafwise reduced cohomology). Since $\phi$ is foliated, there are induced actions $\phi^*=\{\phi^{t*}\}$ on $C^\infty(M;\Lambda\FF)$ and $\bar H^*(\FF)$. In this case, $\FF$ is Riemannian, and therefore it has a leafwise Hodge decomposition \cite{AlvKordy2001}, 
\begin{equation}\label{leafwise Hodge dec, FF Riem, M closed}
C^\infty(M;\Lambda\FF)=\ker\Delta_\FF\oplus\overline{\im d_\FF}\oplus\overline{\im\delta_\FF}\;,
\end{equation}
where $\delta_\FF$ and $\Delta_\FF$ are the leafwise coderivative and leafwise Laplacian. Moreover the leafwise heat operator $e^{-u\Delta_\FF}$ defines a continuous map
\begin{equation}\label{heat op, FF Riem, M closed}
C^\infty(M;\Lambda\FF)\times[0,\infty]\to C^\infty(M;\Lambda\FF)\;,\quad(\alpha,t)\mapsto e^{-u\Delta_\FF}\alpha\;,
\end{equation}
where $\Pi_\FF=e^{-\infty\Delta_\FF}$ is the projection to $\ker\Delta_\FF$ given by~\eqref{leafwise Hodge dec, FF Riem, M closed}. This projection induces a leafwise Hodge isomorphism
\begin{equation}\label{leafwise Hodge iso, FF Riem, M closed}
\bar H^*(\FF)\cong\ker\Delta_\FF\;.
\end{equation}
These properties are rather surprising because the differential complex $d_\FF$ is only leafwise elliptic. Of course, the condition on the foliation to be Riemannian is crucial to make up for the lack of transverse ellipticity. The decomposition~\eqref{leafwise Hodge dec, FF Riem, M closed} may not be valid for non-Riemannian foliations \cite{DeningerSinghof2001}. 

On the other hand, the action $\phi^*$ on $\bar H^*(\FF)$ satisfies the following properties \cite{AlvKordy2002,AlvKordy2008a}. For all $f\in\Cinftyc(\R)$ and $0<u\le\infty$, the operator
\begin{equation}\label{P_u,f, FF Riem, M closed}
P_{u,f}=\int_\R\phi^{t*}e^{-u\Delta_\FF}f(t)\,dt
\end{equation}
is smoothing, and therefore it is of trace class since $M$ is closed. Moreover its super-trace, $\Trs P_{u,f}$, depends continuously on $f$ and is independent of $u$, and the limit of $\Trs P_{u,f}$ as $u\downarrow0$ gives the expected contribution of the closed orbits. But, by~\eqref{leafwise Hodge iso, FF Riem, M closed} and~\eqref{P_u,f, FF Riem, M closed}, the mapping $f\mapsto\Trs P_{\infty,f}$ can be considered as a distributional version of the super-trace of $\phi^*$ on $\bar H^*(\FF)$; i.e., the Lefschetz distribution $\Ldis(\phi)$, solving the problem in this case. 

We would like to extend the trace formula to the case where $\phi$ has fixed points, which are very relevant in Deninger's program. But their existence prevents the foliation from being Riemannian, except in trivial cases. However the foliations with simple foliated flows have a precise description \cite{AlvKordyLeichtnam-sff}. For example, the $\FF$-saturation of the fixed point set of $\phi$ is a finite union $M^0$ of compact leaves, and the restriction $\FF^1$ of $\FF$ to $M^1=M\sm M^0$ is a Riemannian foliation. Moreover $\FF^1$ has bounded geometry in the sense of \cite{Sanguiao2008,AlvKordyLeichtnam2014} for certain bundle-like metric $g^1$ on $M^1$. Then, instead of $C^\infty(M;\Lambda\FF)$, we consider in \cite{AlvKordyLeichtnam-atffff} the space $I(M,M^0;\Lambda\FF)$ of distributional leafwise forms conormal to $M^0$ (the best possible singularities). This is a complex with the continuous extension of $d_\FF$, and we have a short exact sequence of complexes,
\[
0\to K(M,M^0;\Lambda\FF) \hookrightarrow I(M,M^0;\Lambda\FF)\to J(M,M^0;\Lambda\FF)\to0\;,
\]
where $K(M,M^0;\Lambda\FF)$ is the subcomplex supported in $M^0$, and $J(M,M^0;\Lambda\FF)$ is defined by restriction to $M^1$. A key result of \cite{AlvKordyLeichtnam-atffff} is that we also have a short exact sequence in (reduced) cohomology,
\[
0\to H^*K(\FF) \to \bar H^*I(\FF)\to \bar H^*J(\FF)\to0\;,
\]
with corresponding actions $\phi^*=\{\phi^{t*}\}$ induced by $\phi$. Thus we can now define $\Ldis(\phi)=\LdisK(\phi)+\LdisJ(\phi)$, using distributional versions of the super-traces of $\phi^*$ on $H^*K(\FF)$ and $\bar H^*J(\FF)$.

On the one hand, $H^*K(\FF)$ can be described using Novikov cohomologies on $M^0$. Under some conditions and taking coefficients in the normal density bundle, we can define $\LdisK(\phi)$ in this way, with the expected contribution from the fixed points. 

On the other hand, $\bar H^*J(\FF)$ can be described using the reduced cohomology $\bar H^*H^\infty(\FF^1)$ of the cochain complex defined by $d_{\FF^1}$ on the Sobolev space $H^\infty(M^1;\Lambda\FF^1)$ (defined with $g^1$); actually, leafwise Novikov versions of this complex are also needed. At this point, to define $\LdisJ(\phi)$, we need a generalization of~\eqref{leafwise Hodge dec, FF Riem, M closed}--\eqref{P_u,f, FF Riem, M closed} for Riemannian foliations of bounded geometry using this type of cochain complex. This generalization is the purpose of this paper.

Precisely, let $\FF$ be a Riemannian foliation of bounded geometry on an open manifold $M$ with a bundle-like metric. Then Sanguiao \cite{Sanguiao2008} proved versions of~\eqref{leafwise Hodge dec, FF Riem, M closed}--\eqref{leafwise Hodge iso, FF Riem, M closed} using $H^\infty(M;\Lambda\FF)$ and $\bar H^*H^\infty(\FF)$ instead of $C^\infty(M;\Lambda\FF)$ and $\bar H^*(\FF)$. We explain again the proof in terms of our study of Riemannian foliations of bounded geometry \cite{AlvKordyLeichtnam2014}. 

Moreover let $\phi$ be a simple foliated flow on $M$ transverse to the leaves. If the infinitesimal generator of $\phi$ is $C^\infty$ uniformly bounded, then we also get that~\eqref{P_u,f, FF Riem, M closed} defines a smoothing operator $P_{u,f}$, and whose Schwartz kernel is described for $0<u<\infty$. But now the operators $P_{u,f}$ are not of trace class because $M$ is not compact. So additional tools will be used in \cite{AlvKordyLeichtnam-atffff} to define and study $\LdisJ(\phi)$ (see also \cite{KordyukovPavlenko2015}).

Finally, we show how to extend these results to leafwise versions of the Novikov complex, as needed in \cite{AlvKordyLeichtnam-atffff}.

\section{Preliminaries on section spaces and differential operators}\label{s: prelim on section sps and opers}

Let us recall some analytic concepts and fix their notation.

\subsection{Distributional sections}\label{ss: dis sections}

Let $M$ be a (smooth, i.e., $C^\infty$) manifold of dimension $n$, and let $E$ be a (smooth complex) vector bundle over $M$. The space of smooth sections, $C^\infty(M;E)$, is equipped with the (weak) $C^\infty$ topology (see e.g.\ \cite{Hirsch1976}). This notation will be also used for the space of smooth sections of other types of fiber bundles. If we consider only compactly supported sections, we get the space $\Cinftyc(M;E)$, with the compactly supported $C^\infty$ topology. 

Let $\Omega^aE$ ($a\in\R$) denote the line bundle of $a$-densities of $E$, and let $\Omega E=\Omega^1E$. Let $TM$ and $T^*M$ the (complex) tangent and cotangent vector bundles, $\Lambda M=\bigwedge T^*M$, $\Omega^aM=\Omega^aTM$ and $\Omega M=\Omega^1M$.  Moreover let $\fX(M)=C^\infty(M;TM)$ and $\fX_{\text{\rm c}}(M)=\Cinftyc(M;TM)$. The restriction of vector bundles to any submanifold $L\subset M$ may be denoted with a subindex, like $E_L$, $T_LM$, $T^*_LM$ and $\Omega^a_LM$. Redundant notation will be removed; for instance, $C^\infty(L;E)$ and $C^\infty(M;\Omega^a)$ will be used instead of $C^\infty(L;E_L)$ and $C^\infty(M;\Omega^aM)$. We may also use the notation $C^\infty(E)=C^\infty(M;E)$ and $C^\infty_{\text{\rm c}}(E)=C^\infty_{\text{\rm c}}(M;E)$ if there is no danger of confusion. As usual, the trivial line bundle is omitted from this notation: the spaces of smooth (complex) functions and its compactly supported version are denoted by $C^\infty(M)$ and $C^\infty_{\text{\rm c}}(M)$. 

A similar notation is used for other section spaces. For instance, consider also the spaces of distributional (or generalized) sections of $E$, and its compactly supported version,
\[
C^{-\infty}(M;E)=\Cinftyc(M;E^*\otimes\Omega)'\;,\quad
C^{-\infty}_{\text{\rm c}}(M;E)=C^\infty(M;E^*\otimes\Omega)'\;,
\]
where we take the topological\footnote{This term is added to algebraic concepts on topological vector spaces to mean that they are compatible with the topologies. For instance, the \emph{topological dual} $V'$ consists of continuous linear maps $V\to\C$, an isomorphism is called \emph{topological} if it is also a homeomorphism, and a direct sum is called \emph{topological} if it has the product topology.} dual spaces with the weak-$*$ topology. A continuous injection $C^\infty(M;E)\subset C^{-\infty}(M;E)$ is defined by $\langle u,v\rangle=\int_Muv$ for $u\in C^\infty(M;E)$ and $v\in\Cinftyc(M;E^*\otimes\Omega)$, using the canonical pairing of $E$ and $E^*$. There is a similar continuous injection $\Cinftyc(M;E)\subset C^{-\infty}_{\text{\rm c}}(M;E)$. If $E$ is endowed with a Hermitian structure, we can also consider the Banach space $L^\infty(M;E)$ of its essentially bounded sections, whose norm is denoted by $\|\cdot\|_{L^\infty}$. If $M$ is compact, then the equivalence class of $\|\cdot\|_{L^\infty}$ is independent of the Hermitian structure. Also, for any\footnote{We use the notation $\N=\Z^+$ and $\N_0=\N\cup\{0\}$.} $m\in\N_0$, $C^m(M;E)$ denotes the space of $C^m$ sections.

When explicitly indicated, we will also consider real objects with the same notation: real vector bundles, Euclidean structures, real tangent vectors and vector fields, real densities, real functions and distributions, etc.

\subsection{Operators on section spaces}\label{ss: ops}

Let $E$ and $F$ be vector bundles over $M$, and let $A:\Cinftyc(M;E)\to C^\infty(M;F)$ be a continuous linear operator. The transpose of $A$,
\[
A^{\text{\rm t}}:C^{-\infty}_{\text{\rm c}}(M;F^*\otimes\Omega)\to C^{-\infty}(M;E^*\otimes\Omega)
\]
is given by $\langle A^{\text{\rm t}}v,u\rangle=\langle v,Au\rangle$ for $u\in\Cinftyc(M;E)$ and $v\in C^{-\infty}_{\text{\rm c}}(M;F^*\otimes\Omega)$. For instance, the transpose of the continuous dense injection $\Cinftyc(M;E^*\otimes\Omega)\subset C^\infty(M;E^*\otimes\Omega)$ is the continuous dense injection $C^{-\infty}_{\text{\rm c}}(M;E)\subset C^{-\infty}(M;E)$. If there is a restriction $A^{\text{\rm t}}:\Cinftyc(M;F^*\otimes\Omega)\to C^\infty(M;E^*\otimes\Omega)$, then $A^{\text{\rm tt}}:C^{-\infty}_{\text{\rm c}}(M;E)\to C^{-\infty}(M;F)$ is a continuous extension of $A$, also denoted by $A$. The Schwartz kernel, $K_A\in C^{-\infty}(M^2;F\boxtimes(E^*\otimes\Omega))$, is determined by the condition $\langle K_A,v\otimes u\rangle = \langle v,Au\rangle$ for $u\in\Cinftyc(M;E)$ and $v\in\Cinftyc(M;F^*\otimes\Omega)$. The mapping $A\mapsto K_A$ defines a bijection (the Schwartz kernel theorem)\footnote{For locally convex (topological vector) spaces $X$ and $Y$, the notation $L(X,Y)$ is used for the space of continuous linear operators $X\to Y$ with the topology of bounded convergence. $\End(X):=L(X,X)$ is an associative algebra with the operation of composition.} 
\[
L(\Cinftyc(M;E),C^{-\infty}(M;F))\to C^{-\infty}(M^2;F\boxtimes (E^*\otimes\Omega))\;.
\]
Note that
\[
K_{A^{\text{\rm t}}}=R^*K_A\in C^{-\infty}(M^2;(E^*\otimes\Omega)\boxtimes F)\;,
\]
where $R: M^2\to M^2$ is given by $R(x,y)=(y,x)$.

There are obvious versions of the construction of $A^{\text{\rm t}}$ and $A^{\text{\rm tt}}$ when both the domain and target of $A$ have compact support, or no support restriction.

\subsection{Differential operators}\label{ss: diff ops}

Let $\Diff(M)\subset\End(C^\infty(M))$ be the $C^\infty(M)$-submodule and subalgebra of differential operators, filtered by the order. Every $\Diff^m(M)$ ($m\in\N_0$) is $C^\infty(M)$-spanned by all compositions of up to $m$ tangent vector fields, where $\fX(M)$ is considered as the Lie algebra of derivations of $C^\infty(M)$. In particular, $\Diff^0(M)\equiv C^\infty(M)$. Any $A\in\Diff^m(M)$ has the following local description. Given a chart $(U,x)$ of $M$ with $x=(x^1,\dots,x^n)$, let $\partial_j=\frac{\partial}{\partial x^j}$ and $D_j=\frac{1}{i}\partial_j$. For any multi-index $I=(i_1,\dots,i_n)\in\N_0^n$, let $\partial_I=\partial_1^{i_1}\cdots\partial_n^{i_n}$, $D^I=D_x^I=D_1^{i_1}\cdots D_n^{i_n}$ and $|I|=i_1+\dots+i_n$. Then $A=\sum_{|I|\le m}a_ID^I$ on $\Cinftyc(U)$ for some local coefficients $a_I\in C^\infty(U)$.

On the other hand, let $P(T^*M)\subset C^\infty(T^*M)$ be the graded $C^\infty(M)$-module and subalgebra of functions on $T^*M$ whose restriction to the fibers are polynomials, with the grading defined by the degree of the polynomials. In particular, $P^{[0]}(T^*M)\equiv C^\infty(M)$ and $P^{[1]}(T^*M)\equiv\fX(M)$. The principal symbol of any $X\in\fX(M)\subset\Diff^1(M)$ is $\sigma_1(X)=iX\in P^{[1]}(T^*M)$. The map $\sigma_1$ can be extended to a homomorphism of $C^\infty(M)$-modules and algebras, $\sigma :\Diff(M)\to P(T^*M)$, obtaining for every $m$ the principal symbol surjection
\begin{equation}\label{sigma_m}
\sigma_m:\Diff^m(M)\to P^{[m]}(T^*M)\;,
\end{equation}
with kernel $\Diff^{m-1}(M)$.

For vector bundles $E$ and $F$ over $M$, the above concepts can be extended by taking the $C^\infty(M)$-tensor product with $C^\infty(M;F\otimes E^*)$, obtaining the filtered $C^\infty(M)$-submodule
\[
\Diff(M;E,F)\subset L(C^\infty(M;E),C^\infty(M;F))\;,
\]
the graded $C^\infty(M)$-submodule
\[
P^{[m]}(T^*M;F\otimes E^*)\subset C^\infty(T^*M;\pi^*(F\otimes E^*))\;,
\]
where $\pi:T^*M\to M$ is the vector bundle projection, and the principal symbol surjection
\begin{equation}\label{sigma_m on Diff^m(M;E,F)}
\sigma_m:\Diff^m(M;E,F) \to P^{[m]}(T^*M;\pi^*(F\otimes E^*))\;,
\end{equation}
with kernel $\Diff^{m-1}(M;E,F)$. Using local trivializations of $E$ and $F$, any $A\in\Diff^m(M;E,F)$ has local expressions
\[
A=\sum_{|I|\le m}a_ID^I:\Cinftyc(U,\C^l)\to\Cinftyc(U,\C^{l'})
\]
as above, where $l=\rank E$ and $l'=\rank F$, with local coefficients $a_I\in C^\infty(U;\C^{l'}\otimes\C^{l*})$. If $E=F$, then we use the notation $\Diff(M;E)$, which is also a filtered algebra with the operation of composition. Recall that $A\in\Diff^m(M;E,F)$ is called elliptic if $\sigma_m(A)(p,\xi)\in F_p\otimes E^*_p\equiv\Hom(E_p,F_p)$ is an isomorphism for all $p\in M$ and $0\ne\xi\in T^*_pM$. 

Using integration by parts, it follows that the class of differential operators is closed by transposition. So any $A\in\Diff(M;E,F)$ define continuous linear maps (Section~\ref{ss: ops}),
\[
A:C^{-\infty}(M;E)\to C^{-\infty}(M;F)\;,\quad A:C^{-\infty}_{\text{\rm c}}(M;E)\to C^{-\infty}_{\text{\rm c}}(M;F)\;.
\]

\subsection{Sobolev spaces}\label{ss: Sobolev sps}

The Hilbert space $L^2(M;\Omega^{1/2})$ is the completion $\Cinftyc(M;\Omega^{1/2})$ with the scalar product $\langle u,v\rangle=\int_Mu\bar v$. There is a continuous inclusion $L^2(M;\Omega^{1/2})\subset C^{-\infty}(M;\Omega^{1/2})$.

Suppose first that $M$ is compact. Then $L^2(M;\Omega^{1/2})$ is also a $C^\infty(M)$-module. Thus the Sobolev space of order $m\in\N_0$,
\begin{equation}\label{H^m(M;Omega^1/2)}
H^m(M;\Omega^{\frac{1}{2}})=\{\,u\in L^2(M;\Omega^{\frac{1}{2}})\mid 
\Diff^m(M;\Omega^{\frac{1}{2}})\,u\subset L^2(M;\Omega^{\frac{1}{2}})\,\}\;,
\end{equation}
is also a $C^\infty(M)$-module. In particular, $H^0(M;\Omega^{1/2})=L^2(M;\Omega^{1/2})$. By the elliptic estimate, an elliptic operator $P\in\Diff^1(M;\Omega^{1/2})$ can be used to equip $H^m(M;\Omega^{1/2})$ with the Hilbert space structure defined by
\begin{equation}\label{langle u, v rangle_m}
\langle u,v\rangle_m=\langle(1+P^*P)^mu,v\rangle\;.
\end{equation}
The equivalence class of the corresponding norm $\|\cdot\|_m$ is independent of the choice of $P$. Thus $H^m(M;\Omega^{1/2})$ is a Hilbertian space with no canonical choice of a scalar product in general. Now, the Sobolev space of order $-m$ is the Hilbertian space
\begin{equation}\label{H^-m(M;Omega^1/2)}
H^{-m}(M;\Omega^{\frac{1}{2}})=H^m(M;\Omega^{\frac{1}{2}})'\equiv\Diff^m(M;\Omega^{\frac{1}{2}})\,L^2(M;\Omega^{\frac{1}{2}})\;.
\end{equation}

For any vector bundle $E$, the $C^\infty(M)$-module $H^{\pm m}(M;E)$ can be defined as the $C^\infty(M)$-tensor product of $H^{\pm m}(M;\Omega^{1/2})$ with $C^\infty(M;E\otimes\Omega^{-1/2})$; in particular, this defines $L^2(M;E)$. A Hermitian structure $(\cdot,\cdot)$ on $E$ and a non-vanishing smooth density $\omega$ on $M$ can be used to define an obvious scalar product $\langle\cdot,\cdot\rangle$ on $L^2(M;E)$. Using moreover an elliptic operator $P\in\Diff^1(M;E)$, we get a scalar product $\langle\cdot,\cdot\rangle_m$ on $H^m(M;E)$ like in~\eqref{langle u, v rangle_m}, with norm $\|\cdot\|_m$. Indeed, this scalar product makes sense on $C^\infty(M;E)$ for any order $m\in\R$, where $(1+P^*P)^m$ is defined by the functional calculus given by the spectral theorem. Then, taking the corresponding completion of $C^\infty(M;E)$, we get the Sobolev space $H^m(M;E)$ of order $m\in\R$. In particular, $H^{-m}(M;E)\equiv H^m(M;E^*\otimes\Omega)'$.

When $M$ is not compact, any choice of $P$, $(\cdot,\cdot)$ and $\omega$ can be used to equip $\Cinftyc(M;E)$ with a scalar product $\langle\cdot,\cdot\rangle_m$ as above, and the corresponding Hilbert space completion can be denoted by $H^m(M;E)$; in particular, this defines $L^2(M;E)=H^0(M;E)$. But now the equivalence class of $\|\cdot\|_m$ (and therefore $H^m(M;E)$) depends on the choices. However their compactly supported and their local versions, $H^m_{\text{\rm c}}(M;E)$ and $H^m_{\text{\rm loc}}(M;E)$, are independent of the choices involved. In particular, we have $L^2_{\text{\rm c}}(M;E)$ and $L^2_{\text{\rm loc}}(M;E)$. The formal adjoint of any differential operator is locally defined like in the compact case.

In any case, the notation $\|\cdot\|_{m,m'}$ (or $\|\cdot\|_m$ if $m=m'$) is used for the induced norm of operators $H^m(M;E)\to H^{m'}(M;E)$ ($m,m'\in\R$). For example, when $M$ is compact, any $A\in\Diff^m(M;E)$ defines a bounded operator $A:H^{m+s}(M;E)\to H^m(M;E)$ for all $s\in\R$. Taking $s=0$, we can consider $A$ as a densely defined linear operator in $L^2(M;E)$ with domain $H^m(M;E)$. Its adjoint $A^*$ in $L^2(M;E)$ is defined by the formal adjoint $A^*\in\Diff^m(M;E)$, which is locally determined using integration by parts. Recall that $A$ is called formally self-adjoint or symmetric if it is equal to its formal adjoint.

\subsection{Differential complexes}\label{ss: diff complexes}

A {\em topological\/} ({\em cochain\/}) {\em complex\/} $(C,d)$ is a cochain complex where $C$ is a graded topological vector space and $d$ is continuous. Then the cohomology $H(C,d)=\ker d/\im d$ has an induced topology, whose maximal Hausdorff quotient, $\bar H(C,d):=H(C,d)/\overline 0\equiv\ker d/\overline{\im d}$, is called the {\em reduced cohomology\/}. The elements in $H(C,d)$ and $\bar H(C,d)$ defined by some $u\in\ker d$ will be denoted by $[u]$ and $\overline{[u]}$, respectively. The continuous cochain maps between topological complexes induce continuous homomorphisms between the corresponding (reduced) cohomologies. \emph{Topological graded differential algebras} can be similarly defined by assuming that their product is continuous.

Recall that a {\em differential complex\/} of {\em order\/} $m$ is a topological complex of the form $(C^\infty(M;E),d)$, where $E=\bigoplus_rE^r$ and $d=\bigoplus_rd_r$, for a finite sequence of differential operators of the same order $m$,
\[
\begin{CD}
C^\infty(M;E^0) @>{d_0}>> C^\infty(M;E^1) @>{d_1}>> \cdots @>{d_{N-1}}>> C^\infty(M;E^N)\;.
\end{CD}
\]
The compactly supported version $(\Cinftyc(M;E),d)$ may be also considered. Negative or decreasing degrees may be also considered without essential change. Such a differential complex is called {\em elliptic\/} if the symbol sequence,
\[
\begin{CD}
0\to E^0_p @>{\sigma_m(d_0)(p,\xi)}>> E^1_p @>{\sigma_m(d_1)(p,\xi)}>> \cdots @>{\sigma_m(d_{N-1})(p,\xi)}>> E^N_p\to0\;,
\end{CD}
\]
is exact for all $p\in M$ and $0\ne\xi\in T^*_pM$.
 
Suppose that every $E^r$ is equipped with a Hermitian structure, and $M$ with a distinguished non-vanishing smooth density. Then the formal adjoint $\delta=d^*$ also defines a differential complex, giving rise to symmetric operators $D=d+\delta$ and $\Delta=D^2=d\delta+\delta d$ (a generalized Laplacian) in the Hilbert space $L^2(M;E)$. The differential complex $d$ is elliptic if and only if the differential complex $\delta$ is elliptic, and if and only if the differential operator $D$ (or $\Delta$) is elliptic. If $d$ is elliptic and $M$ is closed, then $D$ and $\Delta$ have discrete spectra, giving rise to a topological and orthogonal decomposition (a generalized Hodge decomposition)
\begin{equation}\label{Hodge dec}
C^\infty(M;E)=\ker\Delta\oplus\im\delta\oplus\im d\;,
\end{equation}
which induces a topological isomorphism (a Hodge isomorphism)
\begin{equation}\label{Hodge iso}
H(C^\infty(M;E),d)\cong\ker\Delta\;.
\end{equation}
Thus $H(C^\infty(M;E),d)$ is of finite dimension and Hausdorff.

\subsection{Novikov differential complex}\label{ss: Novikov diff complex}

The most typical example of elliptic differential complex is given by the de~Rham derivative $d$ on $C^\infty(M;\Lambda)$, defining the de~Rham cohomology $H^*(M)=H^*(M;\C)$. Suppose that $M$ is endowed with a Riemannian metric $g$, which defines a Hermitian structure on $TM$. Then we have the de~Rham coderivative $\delta=d^*$, and the symmetric operators, $D=d+\delta$ and $\Delta=D^2=d\delta+\delta d$ (the Laplacian).

With more generality, take any closed $\theta\in C^\infty(M;\Lambda^1)$. For the sake of simplicity, assume that $\theta$ is real. Let $V\in\fX(M)$ be determined by $g(V,\cdot)=\theta$, let $\LL_V$ denote the Lie derivative with respect to $V$, and let ${\theta\!\lrcorner}=-(\theta\wedge)^*=-\iota_V$. Then we have the {\em Novikov operators\/} defined by $\theta$, depending on $z\in\C$,
\begin{align*}
d_z&=d+z\,\theta\wedge\;,\quad
\delta_z=d_z^*=\delta-\bar z\,{\theta\!\lrcorner}\;,\\
D_z&=d_z+\delta_z=D+\Re z\,R_\theta+i\,\Im z\,L_\theta\;,\\
\Delta_z&=D_z^2=d_z\delta_z+\delta_zd_z\\
&=\Delta+\Re z\,(\LL_V+\LL_V^*)-i\,\Im z\,(\LL_V-\LL_V^*)+|z|^2|\theta|^2\;,
\end{align*}
where, for $\alpha\in C^\infty(M;\Lambda^rM)$,
\[
L_\theta\alpha=\theta\cdot\alpha=\theta\wedge\alpha+\theta\lrcorner\,\alpha\;,\quad 
R_\theta\alpha=(-1)^r\alpha\cdot\theta=\theta\wedge\alpha-\theta\lrcorner\,\alpha\;.
\]
 The subindex ``$M$'' may be added to this notation if needed. Here, the dot denotes Clifford multiplication defined via the linear identity $\Lambda M\equiv\Cl(T^*M)$. We may write ${\theta\cdot}=L_\theta$. The differential operator $\LL_V+\LL_V^*$ is of order zero, and $\LL_V-\LL_V^*$ is of order one. The differential complex $(C^\infty(M;\Lambda),d_z)$ is elliptic; indeed, it has the same principal symbol as the de~Rham differential complex. Actually, $\Delta_z$ is a generalized Laplacian \cite[Definition~2.2]{BerlineGetzlerVergne2004}, and therefore $D_z$ is a generalized Dirac operator, and $d_z$ is a generalized Dirac complex. The terms {\em Novikov differential complex\/} and {\em Novikov cohomology\/} are used for $(C^\infty(M;\Lambda),d_z)$ and its cohomology, $H_z^*(M)=H_z^*(M;\C)$.

If $\theta$ is exact, say $\theta=dF$ for some $\R$-valued $F\in C^\infty(M)$, then the Novikov operators are called {\em Witten operators\/}; in particular, $d_z=e^{-zF}\,d\,e^{zF}$ and $\delta_z=e^{\bar zF}\,\delta\,e^{-\bar zF}$. Thus the multiplication operator $e^{zF}$ on $C^\infty(M;\Lambda)$ induces an isomorphism $H_z^*(M)\cong H^*(M)$ in this case.

In the general case, the above kind of argument shows that the isomorphism class of $H_z^*(M)$ depends only on $[\theta]\in H^1(M)$. We can also take a regular covering $\pi:\widetilde M\to M$ so that the lift $\tilde\theta=\pi^*\theta$ is exact, say $\tilde\theta=dF$ for some $\R$-valued $F\in C^\infty(\widetilde M)$. Thus we get the Witten derivative $d_{\widetilde M,z}=e^{-zF}\,d_{\widetilde M}\,e^{zF}$ on $C^\infty(\widetilde M;\Lambda)$, which corresponds to the Novikov derivative $d_{M,z}$ on $C^\infty(M;\Lambda)$ via $\pi^*$. 

For any smooth map $\phi:M\to M$, take a lift $\tilde\phi:\widetilde M\to\widetilde M$; i.e., $\pi\tilde\phi=\phi\pi$. Then $\tilde\phi^*_z=e^{-zF}\,\tilde\phi^*\,e^{zF}=e^{z(\tilde\phi^*F-F)}\,\tilde\phi^*$ is an endomorphism of Witten differential complex $(C^\infty(\widetilde M;\Lambda),d_{\widetilde M,z})$, which can be called a {\em Witten perturbation\/} of $\tilde\phi^*$. For all $\gamma\in\Gamma$, we have $T_\gamma^*(\tilde\phi^*F-F)=\tilde\phi^*F-F$, obtaining $T_\gamma^*\tilde\phi^*_z=\tilde\phi^*_zT_\gamma^*$. Therefore $\tilde\phi^*_z$ induces an endomorphism $\phi^*_z$ of Witten differential complex $d_{M,z}$ on $C^\infty(M;\Lambda)$, which can be considered as a {\em Novikov perturbation\/} of $\phi^*$. This $\phi^*_z$ depends on the choice of the lift $\tilde\phi$ of $\phi$. However, any flow $\phi=\{\phi^t\}$ has a unique lift to a flow $\tilde\phi=\{\tilde\phi^t\}$ on $\widetilde M$, giving rise to a canonical choice of $\phi^{t*}_z$, called the {\em Novikov perturbation\/} of $\phi^{t*}$.

If $M$ is oriented, then 
\begin{equation}\label{delta = (-1)^{nr+n+1} star d star}
{\theta\!\lrcorner}=(-1)^{nr+n+1}\star\,{\theta\wedge}\,\star\;,\quad\delta=(-1)^{nr+n+1}\star d\,\star\;,
\end{equation}
on $C^\infty(M;\Lambda^r)$, using the Hodge operator $\star$ on $\Lambda M$. So
\begin{equation}\label{delta_z = (-1)^nr+n+1 star d_-bar z star}
\delta_z=(-1)^{nr+n+1}\star d_{-\bar z}\,\star\;.
\end{equation}

\section{Preliminaries on bounded geometry}\label{s: bd geom}

The concepts recalled here become relevant when $M$ is not compact. Equip $M$ with a Riemannian metric $g$, and let $\nabla$ denote its Levi-Civita connection, $R$ its curvature, and $\inj_M:M\to\R^+$ its injectivity radius function. Suppose that $M$ is connected, obtaining an induced distance function $d$. Actually, in the non-connected case, we can take $d(p,q)=\infty$ if $p$ and $q$ belong to different connected components. Observe that $M$ is complete if $\inf\inj_M>0$. For $r>0$, $p\in M$ and $S\subset M$, let $B(p,r)$ and $\overline B(p,r)$ denote the open and closed $r$-balls centered at $p$, and $\Pen(S,r)$ and le $\overline{\Pen}(S,r)$ denote the open and closed $r$-penumbras of $S$ (defined by the conditions $d(\cdot,S)<r$ and $d(\cdot,S)\le r$, respectively). We may add the subindex ``$M$'' to this notation if needed, or a subindex ``$a$'' if we are referring to a family of Riemannian manifolds $M_a$.

\subsection{Manifolds and vector bundles of bounded geometry}\label{ss: mfds and vector bdls of bd geom}

It is said that $M$ is of {\em bounded geometry\/} if $\inf\inj_M>0$ and $\sup|\nabla^mR|<\infty$ for every $m\in\N_0$. This concept has the following chart description. 

\begin{thm}[Eichhorn \cite{Eichhorn1991}; see also \cite{Roe1988I,Schick1996,Schick2001}]\label{t: mfd of bd geom}
$M$ is of bounded geometry if and only if, for some open ball $B\subset\R^n$ centered at $0$, there are normal coordinates $y_p:V_p\to B$ at every $p\in M$ such that the corresponding Christoffel symbols $\Gamma^i_{jk}$, as a family of functions on $B$ parametrized by $i$, $j$, $k$ and $p$, lie in a bounded set of the Fr\'echet space $C^\infty(B)$. This equivalence holds as well replacing the Christoffel symbols with the metric coefficients $g_{ij}$.
\end{thm}

\begin{rem}\label{r: equi-bounded geometry}
Any non-connected Riemannian manifold of bounded geometry can be considered as a family of Riemannian manifolds (the connected components), which are of \emph{equi-bounded geometry} in the sense that they satisfy the condition of bounded geometry with the same bounds. Conversely, any disjoint union of Riemannian manifolds of equi-bounded geometry is of bounded geometry.
\end{rem}

Assume that $M$ is of bounded geometry and consider the charts $y_p:V_p\to B$ given by Theorem~\ref{t: mfd of bd geom}. The radius of $B$ will be denoted by $r_0$. 

\begin{prop}[{Schick \cite[Theorem~A.22]{Schick1996}, \cite[Proposition~3.3]{Schick2001}}]\label{p: |partial_I(y_q y_p^-1)|}
For every multi-index $I$, the function $|\partial_I(y_qy_p^{-1})|$ is bounded on $y_p(V_p\cap V_q)$, uniformly on $p,q\in M$.
\end{prop}

\begin{prop}[{Shubin \cite[Appendices~A1.2 and~A1.3]{Shubin1992}; see also \cite[Proposition~3.2]{Schick2001}}]\label{p: p_k}
For any $0<2r\le r_0$, there are a subset $\{p_k\}\subset M$ and some $N\in\N$ such that the balls $B(p_k,r)$ cover $M$, and every intersection of $N+1$ sets $B(p_k,2r)$ is empty. Moreover there is a partition of unity $\{f_k\}$ subordinated to the open covering $\{B(p_k,2r)\}$, which is bounded in the Fr\'echet space\footnote{The definition of $\Cinftyub(M)$ is given in Section~\ref{ss: uniform sps}.} $\Cinftyub(M)$. 
\end{prop}

A vector bundle $E$ of rank $l$ over $M$ is said to be of {\em bounded geometry\/} when it is equipped with a family of local trivializations over the charts $(V_p,y_p)$, for small enough $r_0$, with corresponding defining cocycle $a_{pq}:V_p\cap V_q\to\GL(\C,l)\subset\C^{l^2}$, such that, For every multi-index $I$, the function $|\partial_I(a_{pq}y_p^{-1})|$ is bounded on $y_p(V_p\cap V_q)$, uniformly on $p,q\in M$. When referring to local trivializations of a vector bundle of bounded geometry, we always mean that they satisfy the above condition. If the corresponding defining cocycle is valued in $\operatorname{U}(l)$, then $E$ is said to be of {\em bounded geometry\/} as a Hermitian vector bundle.

\begin{ex}\label{ex: bundles of bd geom}
\begin{enumerate}[{\rm(}i\/{\rm)}]

\item\label{i: associated bundle} If $E$ is associated to the principal $\operatorname{O}(n)$-bundle $P$ of orthonormal frames of $M$ and any unitary representation of $\operatorname{O}(n)$, then it is of bounded geometry in a canonical way. In particular, this applies to $TM$ and $\Lambda M$.

\item\label{i: bundle associated to a reduction} The properties of~\eqref{i: associated bundle} can be extended to the case where $E$ is associated to any reduction $Q$ of $P$ with structural group $H\subset\operatorname{O}(n)$, and any unitary representation of $H$. 

\item\label{i: natural operations} The condition of bounded geometry is preserved by operations of vector bundles induced by operations of vector spaces, like dual vector bundles, direct sums, tensor products, exterior products, etc.

\end{enumerate}
\end{ex}

\subsection{Uniform spaces}\label{ss: uniform sps}

For every $m\in\N_0$, a function $u\in C^m(M)$ is said to be {\em $C^m$-uniformly bounded\/} if there is some $C_m\ge0$ with $|\nabla^{m'}u|\le C_m$ on $M$ for all $m'\le m$. These functions form the {\em uniform $C^m$ space\/}\footnote{Here, the subindex ``ub'' is used instead of the common subindex ``b'' to avoid similarities with the notation of b-calculus used in \cite{AlvKordyLeichtnam-atffff}.} $C_{\text{\rm ub}}^m(M)$, which is a Banach space with the norm $\|\cdot\|_{C^m_{\text{\rm ub}}}$ defined by the best constant $C_m$. As usual, the super-index ``$m$'' may be removed from this notation if $m=0$, and we have $C_{\text{\rm ub}}(M)=C(M)\cap L^\infty(M)$. Equivalently, we may take the norm $\|\cdot\|'_{C^m_{\text{\rm ub}}}$ defined by the best constant $C'_m\ge0$ such that $|\partial_I(uy_p^{-1})|\le C'_m$ on $B$ for all $p\in M$ and $|I|\le m$; in fact, by Proposition~\ref{p: |partial_I(y_q y_p^-1)|}, it is enough to consider any subset of points $p$ so that $\{V_p\}$ covers $M$. The {\em uniform $C^\infty$ space\/} is $\Cinftyub(M)=\bigcap_mC_{\text{\rm ub}}^m(M)$, with the inverse limit topology, called {\em uniform $C^\infty$ topology\/}. It consists of the functions $u\in C^\infty(M)$ such that all functions $uy_p^{-1}$ lie in a bounded set of $C^\infty(B)$, which are said to be {\em $C^\infty$-uniformly bounded\/}. Of course, if $M$ is compact, then the $C^m_{\text{\rm ub}}$ topology is just the $C^m$ topology, and the notation $\|\cdot\|_{C^m}$ and $\|\cdot\|'_{C^m}$ is preferred. On the other hand, the definition of uniform spaces with covariant derivative can be also considered for non-complete Riemannian manifolds.

For a Hermitian vector bundle $E$ of bounded geometry over $M$, the {\em uniform $C^m$ space\/} $C_{\text{\rm ub}}^m(M;E)$, of {\em $C^m$-uniformly bounded\/} sections, can be defined by introducing $\|\cdot\|'_{C^m_{\text{\rm ub}}}$ like the case of functions, using local trivializations of $E$ to consider every $uy_p^{-1}$ in $C^m(B,\C^l)$ for all $u\in C^m(M;E)$. Then, as above, we get the {\em uniform $C^\infty$ space\/} $\Cinftyub(M;E)$ of {\em $C^\infty$-uniformly bounded\/} sections, which are the sections $u\in C^\infty(M;E)$ such that all functions $uy_p^{-1}$ define a bounded set of $C^\infty(B;\C^l)$, equipped with the {\em uniform $C^\infty$ topology\/}. In particular, $\fX_{\text{\rm ub}}(M):=\Cinftyub(M;TM)$ is a $\Cinftyub(M)$-submodule and Lie subalgebra of $\fX(M)$. Observe that 
\begin{equation}\label{C^m_ub(M)}
C^m_{\text{\rm ub}}(M)=\{\,u\in C^m(M)\mid\fX_{\text{\rm ub}}(M)\,\overset{(m)}{\cdots}\,\fX_{\text{\rm ub}}(M)\,u\subset L^\infty(M)\,\}\;.
\end{equation}
Let $\fXcom(M)\subset\fX(M)$ be the subset of complete vector fields.

\begin{prop}\label{p: fX_ub(M) subset fX_com(M)}
$\fX_{\text{\rm ub}}(M)\subset\fXcom(M)$.
\end{prop}

\begin{proof}
Let $X\in\fX_{\text{\rm ub}}(M)$. The maximal domain of the local flow $\phi$ of $X$ is an open neighborhood $\Omega$ of $M\times\{0\}$ in $M\times\R$. By the Picard-Lindel\"of Theorem (see e.g.\ \cite[Theorem~II.1.1]{Hartman1964}) and the $C^\infty$-uniform boundedness of $X$, there is some $\alpha>0$ such that $\{p\}\times(-\alpha,\alpha)\subset\Omega$ for all $p\in M$. So $\Omega=M\times\R$ since $\phi$ is a local flow.
\end{proof}

\subsection{Differential operators of bounded geometry}\label{ss: diff ops of bd geom}

Like in Section~\ref{ss: diff ops}, by using $\fX_{\text{\rm ub}}(M)$ and $\Cinftyub(M)$ instead of $\fX(M)$ and $C^\infty(M)$, we get the filtered subalgebra and $\Cinftyub(M)$-submodule $\Diffub(M)\subset\Diff(M)$ of differential operators of {\em bounded geometry\/}. Let $P_{\text{\rm ub}}(T^*M)\subset P(T^*M)$ be the graded subalgebra generated by $P_{\text{\rm ub}}^{[0]}(T^*M)\equiv \Cinftyub(M)$ and $P_{\text{\rm ub}}^{[1]}(T^*M)\equiv \fX_{\text{\rm ub}}(M)$, which is also a graded $\Cinftyub(M)$-submodule. Then~\eqref{sigma_m} restricts to a surjection $\sigma_m:\Diffub^m(M)\to P_{\text{\rm ub}}^{[m]}(T^*M)$ whose kernel is $\Diffub^{m-1}(M)$. These concepts can be extended to vector bundles of bounded geometry $E$ and $F$ over $M$ by taking the $\Cinftyub(M)$-tensor product with $\Cinftyub(M;F\otimes E^*)$, obtaining the filtered $\Cinftyub(M)$-modules $\Diffub(M;E,F)$ (or $\Diffub(M;E)$ if $E=F$) and $P_{\text{\rm ub}}(T^*M;F\otimes E^*)$, and the surjective restriction
\[
\sigma_m:\Diffub^m(M;E,F)\to P_{\text{\rm ub}}^{[m]}(T^*M;F\otimes E^*)
\]
of~\eqref{sigma_m on Diff^m(M;E,F)}, whose kernel is $\Diffub^{m-1}(M;E,F)$. Bounded geometry of differential operators is preserved by compositions and by taking transposes, and by taking formal adjoints in the case of Hermitian vector bundles of bounded geometry; in particular, $\Diffub(M;E)$ is a filtered subalgebra of $\Diff(M;E)$. Using local trivializations of $E$ and $F$ over the charts $(V_p,y_p)$, we get a local description of any element of $\Diffub^m(M;E,F)$ by requiring the local coefficients to define a bounded subset of the Fr\'echet space $C^\infty(B,\C^{l'}\otimes\C^{l*})$, where $l$ and $l'$ are the ranks of $E$ and $F$. Using the norms $\|\cdot\|'_{C^m_{\text{\rm ub}}}$, it easily follows that every $A\in\Diffub^m(M;E,F)$ defines bounded operators $A:C^{m+s}_{\text{\rm ub}}(M;E)\to C^s_{\text{\rm ub}}(M;F)$ ($s\in\N_0$), which induce a continuous operator $A:\Cinftyub(M;E)\to \Cinftyub(M;F)$.

\begin{ex}\label{ex: connections of bd geom}
\begin{enumerate}[{\rm(}i\/{\rm)}]

\item\label{i: connection of the associated bundle} In Example~\ref{ex: bundles of bd geom}~\eqref{i: associated bundle}, the Levi-Civita connection $\nabla$ induces a connection of bounded geometry on $E$, also denoted by $\nabla$. In particular, $\nabla$ itself is of bounded geometry on $TM$, and induces a connection $\nabla$ of bounded geometry on $\Lambda M$. This holds as well for the connection on $E$ induced by any other Riemannian connection of bounded geometry on $TM$. 

\item\label{i: connection of bundle associated to a reduction}  In Example~\ref{ex: bundles of bd geom}~\eqref{i: bundle associated to a reduction}, if a Riemannian connection of bounded geometry on $TM$ is given by a connection on $Q$, then the induced connection on $E$ is of bounded geometry.

\item\label{i: connections induced by natural operations} In Example~\ref{ex: bundles of bd geom}~\eqref{i: natural operations}, bounded geometry of connections is preserved by taking the induced connections in the indicated operations with vector bundles of bounded geometry. 

\item\label{i: d is of boundedgeometry} The standard expression of the de~Rham derivative $d$ on local coordinates shows that it is of bounded geometry, and therefore $\delta$ is also of bounded geometry. 

\end{enumerate}
\end{ex}

Let $E$ and $F$ be Hermitian vector bundles of bounded geometry. Then any unitary connection $\nabla$ of bounded geometry on $E$ can be used to define an equivalent norm $\|\cdot\|_{C^m_{\text{\rm ub}}}$ on every Banach space $C_{\text{\rm ub}}^m(M;E)$, like in the case of $C_{\text{\rm ub}}^m(M)$. 

It is said that $A\in\Diff^m(M;E,F)$ is {\em uniformly elliptic\/} if there is some $C\ge1$ such that, for all $p\in M$ and $\xi\in T^*_pM$,
\[
C^{-1}|\xi|^m\le|\sigma_m(A)(p,\xi)|\le C|\xi|^m\;.
\]
This condition is independent of the choice of the Hermitian metrics of bounded geometry on $E$ and $F$. Any $A\in\Diff^m_{\text{\rm ub}}(M;E,F)$ satisfies the second inequality.

\subsection{Sobolev spaces of manifolds of bounded geometry}\label{ss: Sobolev, bd geom}

For any Hermitian vector bundle $E$ of bounded geometry over $M$, any choice of a uniformly elliptic $P\in\Diff^1_{\text{\rm ub}}(M;E)$, besides the Riemannian density and the Hermitian structure, can be used to define the Sobolev space $H^m(M;E)$ ($m\in\R$) (Section~\ref{ss: Sobolev sps}). Any choice of $P$ defines the same Hilbertian space $H^m(M;E)$, which is a $\Cinftyub(M)$-module. Every $A\in\Diffub^m(M;E,F)$ defines bounded operators $A:H^{m+s}(M;E)\to H^s(M;F)$ ($s\in\R$), which induce continuous maps $A:H^{\pm\infty}(M;E)\to H^{\pm\infty}(M;F)$.

\begin{prop}[{Roe \cite[Proposition~2.8]{Roe1988I}}]\label{p: Sobolev embedding with bd geometry}
If $m'>m+n/2$, then $H^{m'}(M;E)\subset C_{\text{\rm ub}}^m(M;E)$, continuously. Thus $H^\infty(M;E)\subset \Cinftyub(M;E)$, continuously.
\end{prop}

\subsection{Schwartz kernels on manifolds of bounded geometry}\label{ss: Schwartz kernels, bd geom}

Let $E$ and $F$ be Hermitian vector bundles of bounded geometry over $M$.

\begin{prop}[{Roe \cite[Proposition~2.9]{Roe1988I}}]\label{p: Schwartz kernel with bd geometry}
The Schwartz kernel mapping, $A\mapsto K_A$, defines a continuous linear map
\[
L(H^{-\infty}(M;E),H^\infty(M;F))\to\Cinftyub(M;F\boxtimes(E^*\otimes\Omega))\;.
\]
\end{prop}

\begin{rem}\label{r: supp K_A subset r-penumbra <=> supp Au subset r-penumbra for all u}
Let $A\in L(H^{-\infty}(M;E),H^\infty(M;F))$ and $r>0$. Obviously,
\[
\supp K_A\subset\{\,(p,q)\in M^2\mid d(p,q)\le r\,\}
\]
if and only if, for all $u\in H^{-\infty}(M;E)$,
\[
\supp Au\subset\overline{\Pen}(\supp u,r)\;.
\]
\end{rem}

Recall that a function $\psi\in C(\R)$ is called {\em rapidly decreasing\/} if, for all $k\in\N_0$, there is some $C_k\ge0$ so that $|\psi(x)|\le C_k(1+|x|)^{-k}$. They form a Fr\'echet space denoted by $\RR=\RR(\R)$, with the best constants $C_k$ as semi-norms. If $P\in\Diffub^m(M;E)$ is uniformly elliptic and formally self-adjoint, then it is self-adjoint as an unbounded operator in the Hilbert space $L^2(M;E)$, and the functional calculus given by the spectral theorem defines a continuous linear map
\[
\RR\to L(H^{-\infty}(M;E),H^\infty(M;E))\;,\quad\psi\mapsto\psi(P)\;.
\]
Thus the linear map
\begin{equation}\label{psi mapsto K_psi(P)}
\RR\to \Cinftyub(M;E\boxtimes(E^*\otimes\Omega))\;,\quad\psi\mapsto K_{\psi(P)}\;,
\end{equation}
is continuous by Proposition~\ref{p: Schwartz kernel with bd geometry} \cite[Proposition~2.10]{Roe1988I}. 

For any closed real $\theta\in\Cinftyub(M;\Lambda^1)$ and $z\in\C$, we have the corresponding Novikov operators, $D_z\in\Diffub^1(M;\Lambda)$ and $\Delta_z\in\Diffub^2(M;\Lambda)$ (Section~\ref{ss: Novikov diff complex}), which are uniformly elliptic and formally self-adjoint; indeed $D_z$ is a generalized Dirac operator. Thus, for the symmetric hyperbolic equation
\[
\partial_t\alpha_t=iD_z\alpha_t\;,\quad\alpha_0=\alpha\;,
\]
on any open subset of $M$ and with $t$ in any interval containing $0$, any solution satisfies the finite propagation speed property \cite[Proof of Proposition~1.1]{Chernoff1973} (see also  \cite[Theorem~1.4]{CheegerGromovTaylor1982} and \cite[Proof of Proposition~7.20]{Roe1998})
\begin{equation}\label{unit propagation speed}
\supp\alpha_t\subset\Pen(\supp\alpha,|t|)\;.
\end{equation}
In particular, for the {\em Novikov wave operator\/} $e^{itD_z}$ and any $\alpha\in C^\infty(M;\Lambda)$, $\alpha_t=e^{itD_z}\alpha$ satisfies~\eqref{unit propagation speed}. On the other hand, using the expression of the inverse Fourier transform, we get
\begin{equation}\label{psi(D_z)}
\psi(D_z)=(2\pi)^{-1}\int_{-\infty}^\infty e^{i\xi D_z}\hat\psi(\xi)\,d\xi\;,
\end{equation}
According to Remark~\ref{r: supp K_A subset r-penumbra <=> supp Au subset r-penumbra for all u}, it follows from~\eqref{unit propagation speed} and~\eqref{psi(D_z)} that, for all $r>0$,
\begin{equation}\label{supp hat psi subset[-R,R] => supp K_psi(D_z) subset ...}
\supp\hat\psi\subset[-r,r]\Rightarrow\supp K_{\psi(D_z)}\subset\{\,(p,q)\in M^2\mid d(p,q)\le r\,\}\;.
\end{equation}
For $\psi\in\RR$, the operator $\psi(D_z)$ is smoothing and we may use the notation $k_z=k_{\psi,z}=K_{\psi(D_z)}$. We may also use the notation $k_{u,z}=k_{\psi_u,z}$ for any family of functions $\psi_u\in\RR$ depending on a parameter $u$. For instance, for $\psi_u(x)=e^{-ux^2}$ ($u>0$), we get the {\em Novikov heat kernel\/} $k_{u,z}=K_{e^{-u\Delta_z}}$.

\subsection{Maps of bounded geometry}\label{ss: maps of bd geom}

For $a\in\{1,2\}$, let $M_a$ be a Riemannian manifold of bounded geometry, of dimension $n_a$. Consider a normal chart $y_{a,p}:V_{a,p}\to B_a$ at every $p\in M_a$ satisfying the statement of Theorem~\ref{t: mfd of bd geom}. Let $r_a$ denote the radius of $B_a$. For $0<r\le r_a$, let $B_{a,r}\subset\R^{n_a}$ denote the ball centered at the origin with radius $r$. We have $B_a(p,r)=y_{a,p}^{-1}(B_{a,r})$.

A smooth map $\phi:M_1\to M_2$ is said to be of {\em bounded geometry\/} if, for some $0<r<r_1$ and all $p\in M_1$, we have $\phi(B_1(p,r))\subset V_{2,\phi(p)}$, and the compositions $y_{2,\phi(p)}\phi y_{1,p}^{-1}$ define a bounded family of the Fr\'echet space $C^\infty(B_{1,r},\R^{n_2})$. This condition is preserved by composition of maps. The family of smooth maps $M_1\to M_2$ of bounded geometry is denoted by $\Cinftyub(M_1,M_2)$.

For $m\in\N_0$ and $\phi\in \Cinftyub(M_1,M_2)$, using $\|\cdot\|'_{C^m_{\text{\rm ub}}}$, we easily get that $\phi^*$ induces a bounded homomorphism
\begin{equation}\label{phi^* on C_ub^m}
\phi^*:C_{\text{\rm ub}}^m(M_2;\Lambda)
\to C_{\text{\rm ub}}^m(M_1;\Lambda)\;,
\end{equation}
obtaining a continuous homomorphism
\begin{equation}\label{phi^* on C_ub^infty}
\phi^*:\Cinftyub(M_2;\Lambda)\to \Cinftyub(M_1;\Lambda)\;.
\end{equation}

On the other hand, recall that $\phi$ is called {\em uniformly metrically proper\/} if, for any $s\ge0$, there is some $t_s\ge0$ so that, for all $p,q\in M_1$,
\[
d_2(\phi(p),\phi(q)) \le s \Longrightarrow d_1(p,q) \le t_s\;.
\]
For $0<2r\le r_1,r_2$, take $\{p_{1,k}\}\subset M_1$, $\{p_{2,l}\}\subset M_2$ and $N\in\N$ satisfying the statement of Proposition~\ref{p: p_k}. Then $\phi\in C^\infty(M_1,M_2)$ is uniformly metrically proper if and only if there is some $N'\in\N$ such that every set $\phi^{-1}(B_2(p_{2,l},r))$ meets at most $N'$ sets $B_1(p_{1,k},r)$. Using $\|\cdot\|'_m$ on $H^m(M_a;\Lambda)$ (Section~\ref{ss: uniform sps}), it follows that, if $\phi$ is of bounded geometry and uniformly metrically proper, then $\phi^*$ induces a bounded homomorphism
\begin{equation}\label{phi^* on H^m}
\phi^*:H^m(M_2;\Lambda)\to H^m(M_1;\Lambda)
\end{equation}
for all $m$, obtaining induced continuous homomorphisms
\begin{equation}\label{phi^* on H^pm infty}
\phi^*:H^{\pm\infty}(M_2;\Lambda)\to H^{\pm\infty}(M_1;\Lambda)\;.
\end{equation}
If $\phi\in\Diffeo(M_1,M_2)$, and both $\phi$ and $\phi^{-1}$ are of bounded geometry, then $\phi$ is uniformly metrically proper.

\subsection{Smooth families of bounded geometry}\label{ss: families of bd geom}

Consider the notation of Sections~\ref{ss: mfds and vector bdls of bd geom}--\ref{ss: diff ops of bd geom}. Let $T$ be a manifold, and let $\pr_1:M\times T\to M$ denote the first factor projection. A section $u\in C^\infty(M\times T;\pr_1^*E)$ is called a {\em smooth family of smooth sections\/} of $E$ ({\em parametrized\/} by $T$), and we may use the notation $u=\{\,u_t\mid t\in T\,\}\subset C^\infty(M;E)$, where $u_t=u(\cdot,t)$. Its {\em $T$-support\/} is $\overline{\{\,t\in T\mid u_t\ne0\,\}}$. If the $T$-support is compact, then $u$ is said to be {\em $T$-compactly supported\/}. It is said that $u$ is {\em $T$-locally $C^\infty$-uniformly bounded\/} if any $t\in T$ is in some chart $(O,z)$ of $T$ such that the maps $u(y_p\times z)^{-1}$ define a bounded subset of the Fr\'echet space $C^\infty(B\times z(O),\C^l)$, using local trivializations of $E$ over the normal charts $(V_p,y_p)$. 

In particular, we can consider smooth families of $\C$-valued functions, tangent vector fields and sections of $C^\infty(M;F\otimes E^*)$, which can be used to define a {\em smooth family of differential operators\/}, $A=\{\,A_t\mid t\in T\,\}\subset\Diff(M;E,F)$, like in Section~\ref{ss: diff ops}. The {\em $T$-support\/} of $A$ and the property of being {\em $T$-compactly supported\/} is defined like in the case of sections. Adapting Section~\ref{ss: diff ops of bd geom}, if the smooth families of functions, tangent vector fields and sections used to describe $A$ are $T$-locally $C^\infty$-uniformly bounded, then it is said that $A$ is of {\em $T$-local bounded geometry\/}.

On the other hand, with the notation of Section~\ref{ss: maps of bd geom}, a smooth map $\phi:M_1\times T\to M_2$ is called a {\em smooth family of smooth maps\/} $M_1\to M_2$ ({\em parametrized\/} by $T$). It may be denoted by $\phi=\{\,\phi^t\mid t\in T\,\}$, where $\phi^t=\phi(\cdot,t):M_1\to M_2$. It is said that $\phi$ is of {\em $T$-local bounded geometry\/} if every $t\in T$ is in some chart $(O,z)$ of $T$ such that, for some $0<r<r_1$, we have $\phi(B_1(p,r)\times O)\subset V_{2,\phi(p)}$ for all $p\in M_1$, and the compositions $y_{2,\phi(p)}\phi(y_{1,p}\times z)^{-1}$, for $p\in M_1$, define a bounded subset of the Fr\'echet space $C^\infty(B_{1,r}\times z(O),\R^{n_2})$. The composition of smooth families of maps parametrized by $T$ has the obvious sense and preserves the $T$-local bounded geometry condition. In particular, for a flow $\phi=\{\,\phi^t\mid t\in\R\,\}=\{\phi^t\}$ on $M$, it makes sense to consider the $\R$-local bounded geometry condition. The following result complements Proposition~\ref{p: fX_ub(M) subset fX_com(M)}.

\begin{prop}\label{p: X is C^infty-uniformly bd <=> phi^t is of R-local bd geom}
Let $X\in\fXcom(M)$ with flow $\phi$. Then $X\in\fX_{\text{\rm ub}}(M)$ if and only if $\phi$ is of $\R$-local bounded geometry.
\end{prop}

\begin{proof}
The ``if'' part is obvious because $X_p=\phi_*(0_p,\frac{d}{dt}(0))$ for all $p\in M$, where $0_p$ denotes the zero element of $T_pM$. 

Let us prove the ``only if'' part. First, given $0<r<r_0$ and $t_0\in\R$, since $|X|$ is uniformly bounded on $M$, there is some $\epsilon>0$ such that $\phi^t(B(p,r))\subset B(\phi^{t_0}(p),r_0)=V_p$ for all $p\in M$ if $t\in O=(t_0-\epsilon,t_0+\epsilon)$; in particular,  the compositions $y_{\phi^{t_0}(p)}\phi(y_p\times\id_O)^{-1}$, for all $p\in M$, define a bounded subset of $C(B_r\times O,\R^n)$. Then, the argument of the proof of \cite[Theorem~V.3.1]{Hartman1964} shows that the maps $y_{\phi^{t_0}(p)}\phi(y_p\times\id_O)^{-1}$, for all $p\in M$, define a bounded subset of $C^1(B_r\times O,\R^n)$, where $B_r\subset\R^n$ is the ball centered at the origin of radius $r$. Continuing by induction on $m$, it also follows like in the proofs of \cite[Corollary~V.3.2 and Theorem~V.4.1]{Hartman1964} that the maps $y_{\phi^{t_0}(p)}\phi(y_p\times\id_O)^{-1}$, for all $p\in M$, define a bounded subset of $C^m(B_r\times O,\R^n)$ for all $m$.
\end{proof}

\section{Preliminaries on foliations}\label{s: prelims on folns}

Standard references on foliations are \cite{HectorHirsch1981-A,HectorHirsch1983-B,CamachoLinsNeto1985,Godbillon1991,CandelConlon2000-I,CandelConlon2003-II,Walczak2004}.

\subsection{Foliations}\label{ss: folns}

Recall that a ({\em smooth\/}) {\em foliation\/} $\FF$ on manifold $M$, with {\em codimension\/} $\codim\FF=n'$ and {\em dimension\/} $\dim\FF=n''$, can be described by a {\em foliated atlas\/} of $M$, which consists of charts $(U_k,x_k)$ of the smooth structure of $M$, called {\em foliated charts\/} or {\em foliated coordinates\/}, with
\begin{equation}\label{x = (x',x'')}
x_k=(x'_k,x''_k):U_k\to x_k(U_k)=\Sigma_k\times B''_k\subset\R^{n'}\times\R^{n''}\equiv\R^n\;,
\end{equation}
such that $\Sigma_k$ is open in $\R^{n'}$ and $B''_k$ is an open ball in $\R^{n''}$, and the corresponding changes of coordinates are locally of the form
\begin{equation}\label{changes of foliated coordinates}
x_lx_k^{-1}(u,v)=(h_{lk}(u),g_{lk}(u,v))\;.
\end{equation}
The notation
\[
x_k=(x^1_k,\dots,x^n_k)=(x'^1_k,\dots,x'^{n'}_k,x''^{n'+1}_k,\dots,x''^n_k)
\]
will be also used. In the case of codimension one, the notation $(x,y)=(x,y^1,\dots,y^{n-1})$ will be often used instead of $(x',x'')$. It may be also said that $(M,\FF)$ is a {\em foliated manifold\/}. The open sets $U_k$ and the projections $x'_k:U_k\to\Sigma_k$ are said to be {\em distinguished\/}, and the fibers of $x'_k$ are called {\em plaques\/}. The smooth submanifolds $x_k^{\prime\prime-1}(v)\subset U_k$ ($v\in B''_k$) are called {\em local transversals\/} defined by $(U_k,x_k)$, which can be identified with $\Sigma_k$ via $x'_k$. All possible plaques form a base of a finer topology on $M$, becoming a smooth manifold of dimension $n''$ with the obviously induced charts whose connected components are called {\em leaves\/}. The leaf through any point $p$ may be denoted by $L_p$. Foliations on manifolds with boundary are similarly defined, where the boundary is tangent or transverse to the leaves. The {\em $\FF$-saturation\/} of a subset $S\subset M$, denoted by $\FF(S)$, is the union of leaves that meet $S$.

If a smooth map $\phi:M'\to M$ is transverse to (the leaves of) $\FF$, then the connected components of the inverse images $\phi^{-1}(L)$ of the leaves $L$ of $\FF$ are the leaves of a smooth foliation $\phi^*\FF$ on $M'$ of codimension $n'$, called \emph{pull-back} of $\FF$ by $\phi$. In particular, for the inclusion map of any open subset, $\iota:U\hookrightarrow M$, the pull-back $\iota^*\FF$ is the \emph{restriction} $\FF|_U$, which can be defined by the charts of $\FF$ with domain in $U$.

Given foliations $\FF_a$ on manifolds $M_a$ ($a=1,2$), the products of leaves of $\FF_1$ and $\FF_2$ are the leaves of the {\em product foliation\/} $\FF_1\times\FF_2$, whose charts can be defined using products of charts of $\FF_1$ and $\FF_2$. Any connected manifold $M'$ can be considered as a foliation with one leaf. We can also consider the foliation by points on $M'$, denoted by $M'_{\text{\rm pt}}$. Thus we get the foliations $\FF\times M'$ and $\FF\times M'_{\text{\rm pt}}$ on $M\times M'$.

\subsection{Holonomy}\label{ss: holonomy}

After considering a possible refinement, we can assume that the foliated atlas $\{U_k,x_k\}$ is {\em regular\/} in the following sense: it is locally finite; for every $k$, there is a foliated chart $(\widetilde U_k,\tilde x_k)$ such that $\overline{U_k}\subset\widetilde U_k$ and $\tilde x_k$ extends $x_k$; and, if $U_{kl}:=U_k\cap U_l\ne\emptyset$, then there is another foliated chart $(U,x)$ such that $\overline{U_k}\cup\overline{U_l}\subset U$. In this case,~\eqref{changes of foliated coordinates} holds on the whole of $U_{kl}$, obtaining diffeomorphisms $h_{kl}:x'_l(U_{kl})\to x'_k(U_{kl})$ determined by the condition $h_{kl}x'_l=x'_k$ on $U_{kl}$, called {\em elementary holonomy transformations\/}. The collection $\{U_k,x'_k,h_{kl}\}$ is called a {\em defining cocycle\/}. The elementary holonomy transformations $h_{kl}$ generate the so called {\em holonomy pseudogroup\/} $\HH$ on $\Sigma:=\bigsqcup_k\Sigma_k$, which is unique up to certain {\em equivalence\/} of pseudogroups \cite{Haefliger1980}. The $\HH$-orbit of every $\bar p\in\Sigma$ is denoted by $\HH(\bar p)$. The maps $x'_k$ define a homeomorphism between the leaf space, $M/\FF$, and the orbit space, $\Sigma/\HH$.

The paths in the leaves are called {\em leafwise paths\/} when considered in $M$. Let $c:I:=[0,1]\to M$ be a leafwise path with $p:=c(0)\in U_k$ and $q:=c(1)\in U_l$, and let $\bar p=x'_k(p)\in\Sigma_k$ and $\bar q=x'_k(q)\in\Sigma_l$. There is a partition of $I$, $0=t_0<t_1<\dots<t_m=1$, and a sequence of indices, $k=k_1,k_2,\dots,k_m=l$, such that $c([t_{i-1},t_i])\subset U_{k_i}$ for $i=1,\dots,m$. The composition $h_c=h_{k_mk_{m-1}}\cdots h_{k_2k_1}$ is a diffeomorphism with $\bar p\in\dom h_c\subset\Sigma_k$ and $\bar q=h_c(\bar p)\in\im h_c\subset\Sigma_l$. The tangent map $h_{c*}:T_{\bar p}\Sigma_k\to T_{\bar q}\Sigma_l$ is called {\em infinitesimal holonomy\/} of $c$. The germ $\bfh_c$ of $h_c$ at $\bar p$, called {\em germinal holonomy\/} of $c$, depends only on $\FF$ and the end-point homotopy class of $c$ in $L:=L_p$.

\subsection{Infinitesimal transformations and transverse vector fields}\label{ss: infinitesimal transfs}

The vectors tangent to the leaves form the {\em tangent bundle\/} $T\FF\subset TM$, obtaining also the {\em normal bundle\/} $N\FF=TM/T\FF$, the  {\em cotangent bundle\/} $T^*\FF=(T\FF)^*$ and the {\em conormal bundle\/} $N^*\FF=(N\FF)^*$, the {\em tangent/normal density bundles\/}, $\Omega^a\FF=\Omega^aT\FF$ ($a\in\R$) and $\Omega^aN\FF$ (removing ``$a$'' from the notation when it is $1$), and the {\em tangent/normal exterior bundles\/}, $\Lambda\FF=\bigwedge T^*\FF$ and $\Lambda N\FF=\bigwedge N^*\FF$. The complex versions of these vector bundles are taken, unless it is explicitly indicated that the real versions are considered. The terms {\em tangent/normal\/} vector fields, densities and differential forms are used for their smooth sections. Sometimes, the terms ``{\em leafwise\/}'' or ``{\em vertical\/}'' are used instead of ``tangent''. By composition with the canonical projection $TM\to N\FF$, any $X$ in $TM$ or $\fX(M)$ defines an element of $N\FF$ or $C^\infty(M;N\FF)$ denoted by $\overline{X}$. For any smooth local transversal $\Sigma$ of $\FF$ through a point $p\in M$, there is a canonical isomorphism $T_p\Sigma\cong N_p\FF$.

A smooth vector bundle $E$ over $M$, endowed with a flat $T\FF$-partial connection, is said to be {\em$\FF$-flat\/}. For instance, $N\FF$ is $\FF$-flat with the $T\FF$-partial connection $\nabla^\FF$ given by $\nabla^\FF_V\overline{X}=\overline{[V,X]}$ for $V\in\fX(\FF):=C^\infty(M;T\FF)$ and $X\in\fX(M)$. For every leafwise path $c$ from $p$ to $q$, its infinitesimal holonomy can be considered as a homomorphism $h_{c*}:N_p\FF\to N_q\FF$, which equals the $\nabla^\FF$-parallel transport along $c$.

$\fX(\FF)$ is a Lie subalgebra and $C^\infty(M)$-submodule of $\fX(M)$, whose normalizer is denoted by $\fX(M,\FF)$, obtaining the quotient Lie algebra $\olfX(M,\FF)=\fX(M,\FF)/\fX(\FF)$. The elements of $\fX(M,\FF)$ (respectively, $\olfX(M,\FF)$) are called {\em infinitesimal transformations\/} (respectively, {\em transverse vector fields\/}) of $(M,\FF)$. The projection of every $X\in\fX(M,\FF)$ to $\olfX(M,\FF)$ is also denoted by $\overline{X}$; in fact, $\olfX(M,\FF)$ can be identified with the linear subspace of $C^\infty(M;N\FF)$ consisting of the $\nabla^\FF$-parallel normal vector fields (those that are invariant by infinitesimal holonomy). Any $X\in\fX(M)$ is in $\fX(M,\FF)$ if and only if every restriction $X|_{U_k}$ can be projected by $x'_k$, defining an $\HH$-invariant vector field on $\Sigma$, also denoted by $\overline X$. This induces a canonical isomorphism of $\olfX(M,\FF)$ to the Lie algebra $\fX(\Sigma,\HH)$ of $\HH$-invariant tangent vector fields on $\Sigma$.

When $M$ is not closed, we can consider the subsets of complete vector fields, $\fXcom(\FF)\subset\fX(\FF)$ and $\fXcom(M,\FF)\subset\fX(M,\FF)$. Let $\olfXcom(M,\FF)\subset\olfX(M,\FF)$ be the projection of $\fXcom(M,\FF)$.

\subsection{Holonomy groupoid}\label{ss: holonomy groupoid}

On the space of leafwise paths in $M$, with the compact-open topology, two leafwise paths are declared to be equivalent if they have the same end points and the same germinal holonomy. This is an equivalence relation, and the corresponding quotient space, $\fG=\Hol(M,\FF)$, becomes a smooth manifold of dimension $n+n''$ in the following way. An open neighborhood $\fU$ of a class $[c]$ in $\fG$, with $c(0)\in U_k$ and $c(1)\in U_l$, is defined by the leafwise paths $d$ such that $d(0)\in U_k$, $d(1)\in U_l$, $x'_kd(0)\in\dom h_c$, and $h_d$ and $h_c$ have the same germ at $x'_kd(0)$. Local coordinates on $\fU$ are given by $[d]\mapsto(x_kd(0),x''_ld(1))$. Moreover $\fG$ is a Lie groupoid, called the {\em holonomy groupoid\/}, where the space of units $\fG^{(0)}\equiv M$ is defined by the constant paths, the source and range projections $\bfs,\bfr:\fG\to M$ are given by the first and last points of the paths, and the operation is induced by the opposite of the usual path product\footnote{A product of leafwise paths, $c_2\cdot c_1$, is defined if $c_1(1)=c_2(0)$, and it is equal to $c_1$ followed by $c_2$, reparametrized in the usual way.}. Note that $\fG$ is Hausdorff if and only if $\HH$ is {\em quasi-analytic\/} in the sense that, for any $h\in\HH$ and every open $O\subset\Sigma$ with $\overline O\subset\dom h$, if $h|_O=\id_O$, then $h$ is the identity on some neighborhood of $\overline O$. Observe also that $\bfs,\bfr:\fG\to M$ are smooth submersions, and $(\bfr,\bfs):\fG\to M^2$ is a smooth immersion. Let $\RR_\FF=\{\,(p,q)\in M^2\mid L_p=L_q\,\}\subset M^2$, which is not a regular submanifold in general, and let $\Delta\subset M^2$ be the diagonal. We have $(\bfr,\bfs)(\fG)=\RR_\FF$ and $(\bfr,\bfs)(\fG^{(0)})=\Delta$. For any leaf $L$ and $p\in L$, we have $\Hol(L,p)=\bfs^{-1}(p)\cap\bfr^{-1}(p)$, the map $\bfr:\bfs^{-1}(p)\to L$ is the covering projection $\widetilde L^{\text{\rm hol}}\to L$, and $\bfs:\bfr^{-1}(p)\to L$ corresponds to $\bfr:\bfs^{-1}(p)\to L$ by the inversion of $\fG$. Thus $(\bfr,\bfs):\fG\to M^2$ is injective if and only if all leaves have trivial holonomy groups, but, even in this case, it may not be a topological embedding. The fibers of $\bfs$ and $\bfr$ define smooth foliations of codimension $n$ on $\fG$. We also have the smooth foliation $\bfs^*\FF=\bfr^*\FF$ of codimension $n'$ with leaves $\bfs^{-1}(L)=\bfr^{-1}(L)=(\bfr,\bfs)^{-1}(L^2)$ for leaves $L$ of $\FF$, and every restriction $(\bfr,\bfs):(\bfr,\bfs)^{-1}(L^2)\to L^2$ is a smooth covering projection.

Let $\FF_k=\FF|_{U_k}$, $\fG_k=\Hol(U_k,\FF_k)$ and $\RR_k=\RR_{\FF_k}$. Then $\bigcup_k\fG_k$ (respectively, $\bigcup_k\RR_k$) is an open neighborhood of $\fG^{(0)}$ in $\fG$ (respectively, of $\Delta$ in $\RR_\FF$). Furthermore, by the regularity of $\{U_k,x_k\}$, the map $(\bfr,\bfs):\bigcup_k\fG_k\to M^2$ is a smooth embedding with image $\bigcup_k\RR_k$; we will write $\bigcup_k\fG_k\equiv\bigcup_k\RR_k$.

\subsection{The convolution algebra on $\fG$ and its global action}\label{ss: global action}

Consider the notation of Section~\ref{ss: holonomy groupoid}. For the sake of simplicity, assume that $\fG$ is Hausdorff. The extension of the following concepts to the case where $\fG$ is not Hausdorff can be made like in \cite{Connes1979}.

Given a vector bundle $E$ over $M$, consider the vector bundle $S=\bfr^*E\otimes \bfs^*(E^*\otimes\Omega\FF)$ over $\fG$. Let $C^\infty_{\text{\rm cs}}(\fG;S)\subset C^\infty(\fG;S)$ denote the subspace of sections $k\in C^\infty(\fG;S)$ such that $\supp k\cap\bfs^{-1}(K)$ is compact for all compact $K\subset M$; in particular, $C^\infty_{\text{\rm cs}}(\fG;S)=\Cinftyc(\fG;S)$ if $M$ is compact. Similarly, define $C^\infty_{\text{\rm cr}}(\fG;S)$ by using $\bfr$ instead of $\bfs$. Both $C^\infty_{\text{\rm cs}}(\fG;S)$ and $C^\infty_{\text{\rm cr}}(\fG;S)$ are associative algebras with the {\em convolution\/} product defined by
\begin{align*}
(k_1*k_2)(\gamma)&=\int_{\delta\epsilon=\gamma}k_1(\delta)\,k_2(\epsilon)\\
&=\int_{\bfs(\epsilon)=\bfs(\gamma)}k_1(\gamma\epsilon^{-1})\,k_2(\epsilon)
=\int_{\bfr(\delta)=\bfr(\gamma)}k_1(\delta)\,k_2(\delta^{-1}\gamma)\;,
\end{align*}
and $C^\infty_{\text{\rm csr}}(\fG;S):=C^\infty_{\text{\rm cs}}(\fG;S)\cap C^\infty_{\text{\rm cr}}(\fG;S)$ and $\Cinftyc(\fG;S)$ are subalgebras. 

There is a {\em global action\/} of $C^\infty_{\text{\rm cs}}(\fG;S)$ on $C^\infty(M;E)$ defined by
\[
(k\cdot u)(p)=\int_{\bfr(\gamma)=p}k(\gamma)\,u(\bfs(\gamma))\;.
\]
In this way, $C^\infty_{\text{\rm cs}}(\fG;S)$ can be understood as an algebra of operators on $C^\infty(M;E)$. Moreover $C^\infty_{\text{\rm csr}}(\fG;S)$ preserves $\Cinftyc(M;E)$, obtaining an algebra of operators on $\Cinftyc(M;E)$. It can be said that these operators are defined by a leafwise version of the Schwartz kernel (cf.\ Section~\ref{ss: ops}). 

Let $S'$ be defined like $S$ with $E^*\otimes\Omega\FF$ instead of $E$. Then the mapping $k\mapsto k^{\text{\rm t}}$, where $k^{\text{\rm t}}(\gamma)=k(\gamma^{-1})$, defines anti-homomorphisms $C^\infty_{\text{\rm cs/cr}}(\fG;S)\to C^\infty_{\text{\rm cr/cs}}(\fG;S')$ and $C^\infty_{\text{\rm csr}}(\fG;S)\to C^\infty_{\text{\rm csr}}(\fG;S')$, obtaining a leafwise version of the transposition of operators (cf.\ Section~\ref{ss: ops}). Similarly, using $E=\Omega^{1/2}\FF$, or if $E$ has a Hermitian structure and we fix a non-vanishing leafwise density, we get a leafwise version of the adjointness of operators.

\subsection{Leafwise distance}\label{ss: leafwise distance}

Assume that $M$ is a Riemannian manifold, and consider the induced Riemannian metric on the leaves. The \emph{leafwise distance} is the map $d_\FF:M^2\to[0,\infty]$ given by the distance function of the leaves on $\RR_\FF$, and with $d_\FF(M^2\sm\RR_\FF)=\infty$. Note that $d_\FF\ge d_M$. Given $p\in M$, $S\subset M$ and $r>0$, the {\em open\/} and {\em closed leafwise balls\/}, $B_\FF(p,r)$ and $\overline B_\FF(p,r)$, and the {\em open\/} and {\em closed leafwise penumbras\/}, $\Pen_\FF(S,r)$ and $\overline{\Pen}_\FF(S,r)$, are defined with $d_\FF$ like in the case of Riemannian metrics (see Section~\ref{s: bd geom}).

Equip $\fG$ with the Riemannian structure so that the smooth immersion $(\bfr,\bfs):\fG\to M^2$ is isometric. Let $d_\bfr:\fG\to[0,\infty]$ denote the leafwise distance for the foliation on $\fG$ defined by the fibers of $\bfr$, and consider the corresponding open and closed leafwise penumbras, $\Pen_\bfr(\fG^{(0)},r)$ and $\overline{\Pen}_\bfr(\fG^{(0)},r)$. Note that we get the same penumbras by using $\bfs$ instead of $\bfr$; indeed, they are defined by the conditions $d_\FF^{\text{\rm hol}}<r$ and $d_\FF^{\text{\rm hol}}\le r$, respectively, where $d_\FF^{\text{\rm hol}}:\fG\to[0,\infty)$ is defined by $d_\FF^{\text{\rm hol}}(\gamma)=\inf_c\length(c)$, with $c$ running in the piecewise smooth representatives of $\gamma$.

We have $d_\FF^{\text{\rm hol}}\equiv d_\FF$ on $\bigcup_k\fG_k\equiv\bigcup_k\RR_k$. Using the convexity radius (see e.g.\ \cite[Section~6.3.2]{Petersen1998}), it follows that, after refining $\{U_k,x_k\}$ if necessary, we can assume $d_\FF$ is continuous on $\bigcup_k\RR_k$.

From now on, suppose that the leaves are complete Riemannian manifolds. Then the exponential maps of the leaves define a smooth map, $\exp_\FF:T\FF\to M$, on the real tangent bundle of $\FF$.

\begin{lem}\label{l: overline Pen_FF(Q,r) is compact}
For all compact $Q\subset M$ and $r>0$, $\Pen_\FF(Q,r)$ is relatively compact in $M$, and $\Pen_\bfr(\fG^{(0)},r)\cap\bfs^{-1}(Q)$ and $\Pen_\bfs(\fG^{(0)},r)\cap\bfr^{-1}(Q)$ are relatively compact in $\fG$.
\end{lem}

\begin{proof}
The set $E=\{\,v\in T\FF\mid\|v\|\le r\,\}$ is a subbundle of $T\FF$ with compact typical fiber, $\overline B_{\R^{n''}}(0,r)$. So its restriction $E_Q$ is compact, obtaining that $\exp_\FF(E_Q)=\overline{\Pen}_\FF(Q,r)$ is compact. 

For every $v\in T\FF$, let $c_v:I\to M$ denote the leafwise path defined by $c_v(t)=\exp_\FF(tv)$. A smooth map $\sigma:T\FF\to\fG$ is defined by $\sigma(v)=[c_v]$. We get that $\sigma(E_Q)=\overline{\Pen}_\FF(\fG^{(0)},r)\cap\bfs^{-1}(Q)$ is compact, as well as $\overline{\Pen}_\bfs(\fG^{(0)},r)\cap\bfr^{-1}(Q)=(\overline{\Pen}_\bfr(\fG^{(0)},r)\cap\bfs^{-1}(Q))^{-1}$.
\end{proof}

With the notation of Section~\ref{ss: global action}, let $C^\infty_{\text{\rm p}}(\fG;S)\subset C^\infty(\fG;S)$ denote the subspace of sections supported in leafwise penumbras of $\fG^{(0)}$. By Lemma~\ref{l: overline Pen_FF(Q,r) is compact}, this is a subalgebra of $C^\infty_{\text{\rm csr}}(\fG;S)$, and the leafwise transposition restricts to an anti-homomorphism $C^\infty_{\text{\rm p}}(\fG;S)\to C^\infty_{\text{\rm p}}(\fG;S')$.

\subsection{Foliated maps and foliated flows}\label{ss: fol maps}

A {\em foliated map\/} $\phi:(M_1,\FF_1)\to(M_2,\FF_2)$ is a map $\phi:M_1\to M_2$ that maps leaves of $\FF_1$ to leaves of $\FF_2$. In this case, assuming that $\phi$ is smooth, its tangent map defines morphisms $\phi_*:T\FF_1\to T\FF_2$ and $\phi_*:N\FF_1\to N\FF_2$, where the second one is compatible with the corresponding flat partial connections. We also get an induced Lie groupoid homomorphism $\Hol(\phi):\Hol(M_1,\FF_1)\to\Hol(M_2,\FF_2)$, defined by $\Hol(\phi)([c])=[\phi c]$. The set of smooth foliated maps $(M_1,\FF_1)\to(M_2,\FF_2)$ is denoted by $C^\infty(M_1,\FF_1;M_2,\FF_2)$.  A smooth family $\phi=\{\,\phi^t\mid t\in T\,\}$ of foliated maps $(M_1,\FF_1)\to(M_2,\FF_2)$ can be considered as the smooth foliated map $\phi:(M_1\times T,\FF_1\times T_{\text{\rm pt}})\to(M_2,\FF_2)$.

For example, given another manifold $M'$, if a smooth map $\psi:M'\to M$ is transverse to $\FF$, then it is a foliated map $(M',\psi^*\FF)\to(M,\FF)$.

Let $\Diffeo(M,\FF)$ be the group of foliated diffeomorphisms (or transformations) of $(M,\FF)$. A smooth flow $\phi=\{\phi^t\}$ on $M$ is called {\em foliated\/} if $\phi^t\in\Diffeo(M,\FF)$ for all $t\in\R$. More generally, a local flow $\phi:\Omega\to M$, defined on some open neighborhood $\Omega$ of $M\times\{0\}$ in $M\times\R$, is called {\em foliated\/} if it is a foliated map $(\Omega,(\FF\times\R_{\text{\rm pt}})|_\Omega)\to(M,\FF)$. Then $\fX(M,\FF)$ consists of the smooth vector fields whose local flow is foliated, and $\fXcom(M,\FF)$ consists of the complete smooth vector fields whose flow is foliated.

Let $X\in\fXcom(M,\FF)$, with foliated flow $\phi=\{\phi^t\}$. With the notation of Sections~\ref{ss: holonomy} and~\ref{ss: infinitesimal transfs}, let $\bar\phi$ be the local flow on $\Sigma$ generated by $\overline X\in\fX(\Sigma,\HH)$. Since $X|_{U_k}$ corresponds to $\overline X|_{\Sigma_k}$ via $x'_k:U_k\to\Sigma_k$, the local flow defined by $\phi$ on every $U_k$ also corresponds to the restriction of $\bar\phi$ to $\Sigma_k$. Hence $\bar\phi$ is $\HH$-equivariant in an obvious sense.

\subsection{Differential operators on foliated manifolds}\label{ss: diff ops on fol mfds}

Like in Section~\ref{ss: diff ops}, using $\fX(\FF)$ instead of $\fX(M)$, we get the filtered $C^\infty(M)$-submodule and subalgebra of {\em leafwise differential operators\/}, $\Diff(\FF)\subset\Diff(M)$, and a {\em leafwise principal symbol\/} surjection for every order $m$,
\[
\Fsigma_m:\Diff^m(\FF)\to P^{[m]}(T^*\FF)\to0\;,
\]
whose kernel is $\Diff^{m-1}(\FF)$. Moreover these concepts can be extended to vector bundles $E$ and $F$ over $M$ like in Section~\ref{ss: diff ops}, obtaining the filtered $C^\infty(M)$-submodule $\Diff(\FF;E,F)$ (or $\Diff(\FF;E)$ if $E=F$) of $\Diff(M;E,F)$, and the {\em leafwise principal symbol\/} surjection
\[
\Fsigma_m:\Diff^m(\FF;E,F)\to P^{[m]}(T^*\FF;F\otimes E^*)\;,
\]
whose kernel is $\Diff^{m-1}(\FF;E,F)$. The diagram
\begin{equation}\label{CD, Fsigma_m, sigma_m}
\begin{CD}
\Diff^m(\FF;E,F) @>{\Fsigma_m}>> P^{[m]}(T^*\FF;F\otimes E^*) \\
@VVV @VVV \\
\Diff^m(M;E,F) @>{\sigma_m}>> P^{[m]}(T^*M;F\otimes E^*)
\end{CD}
\end{equation}
is commutative, where the left-hand side vertical arrow denotes the inclusion homomorphism, and the right-hand side vertical arrow is defined by the restriction morphism $T^*M\to T^*\FF$. The condition of being a leafwise differential operator is preserved by compositions and by taking transposes, and by taking formal adjoints in the case of Hermitian vector bundles; in particular, $\Diff(\FF;E)$ is a filtered subalgebra of $\Diff(M;E)$.  It is said that $A\in\Diff^m(\FF;E,F)$ is {\em leafwisely elliptic\/} if the leafwise symbol $\Fsigma_m(A)(p,\xi)$ is an isomorphism for all $p\in M$ and $0\ne\xi\in T^*_p\FF$.

A smooth family of leafwise differential operators, $A=\{\,A_t\mid t\in T\,\}$ with $A_t\in\Diff^m(\FF;E,F)$, can be canonically considered as a leafwise differential operator $A\in\Diff^m(\FF\times T_{\text{pt}};\pr_1^*E,\pr_1^*F)$, where $\pr_1:M\times T\to M$ is the first factor projection.

On the other hand, using the canonical injection $N^*\FF\subset T^*M$, it is said that $A\in\Diff^m(M;E,F)$ is {\em transversely elliptic\/} if the symbol $\sigma_m(A)(p,\xi)$ is an isomorphism for all $p\in M$ and $0\ne\xi\in N^*_p\FF$.

\subsection{Riemannian foliations}\label{ss: Riem folns}

The $\HH$-invariant structures on $\Sigma$ are called ({\em invariant\/}) {\em transverse structures\/}. For instance, we will use the concepts of a {\em transverse orientation\/}, a {\em transverse Riemannian metric\/}, and a {\em transverse parallelism\/}. The existence of these transverse structures defines the classes of {\em transversely orientable\/}, ({\em transversely\/}) {\em Riemannian\/}, and {\em transversely parallelizable\/} ({\em TP\/}) foliations. If a transverse parallelism of $\FF$ is a base of a Lie subalgebra $\fg\subset\fX(\Sigma,\HH)$, it gives rise to the concepts of {\em transverse Lie structure\/} and ($\fg$-){\em Lie foliation\/}. If $G$ is the simply connected Lie group with Lie algebra $\fg$, then $\FF$ is a $\fg$-Lie foliation just when $\HH$ is equivalent to some pseudogroup on $G$ generated by some left translations. 

By using the canonical isomorphism $\olfX(M,\FF)\cong\fX(\Sigma,\HH)$, the condition on $\FF$ to be $TP$ means that there is a global frame of $N\FF$ consisting of transverse vector fields $\overline{X_1},\dots,\overline{X_{n'}}$, also called a {\em transverse parallelism\/}; and the condition on $\FF$ to be a $\fg$-Lie foliation means that moreover $\overline{X_1},\dots,\overline{X_{n'}}$ form a base of a Lie subalgebra of $\fg\subset\olfX(M,\FF)$. With this point of view, if moreover $\overline{X_1},\dots,\overline{X_{n'}}\in\olfXcom(M,\FF)$, then the TP or Lie foliation $\FF$ is called {\em complete\/}. 

Similarly, a transverse Riemannian metric can be described as a Euclidean structure on $N\FF$ that is invariant by infinitesimal holonomy. In turn, this is induced by a Riemannian metric on $M$ such that every $x'_k:U_k\to\Sigma_k$ is a Riemannian submersion, called {\em bundle-like metric\/}. Thus $\FF$ is Riemannian if and only if there is a bundle-like metric on $M$. 

It is said that $\FF$ is {\em transitive at\/} a point $p\in M$ when the evaluation map $\ev_p:\fX(M,\FF)\to T_pM$ is surjective, or, equivalently, the evaluation map $\overline{\ev}_p:\olfX(M,\FF)\subset C^\infty(M;N\FF)\to N_p\FF$ is surjective. The transitive point set is open and saturated. The foliation $\FF$ is called {\em transitive\/} if it is transitive at every point. It is said that $\FF$ is {\em transversely complete\/} ({\em TC\/}) if $\ev_p(\fXcom(M,\FF))$ generates $T_pM$ for all $p\in M$. Since the evaluation map $\fXcom(\FF)\to T_p\FF$ is surjective \cite[Section~4.5]{Molino1988}, $\FF$ is TC if and only if $\overline{\ev}_p(\olfXcom(M,\FF))$ generates $N_p\FF$ for all $p\in M$. 

All TP foliations are transitive, and all transitive foliations are Riemannian. On the other hand, Molino's theory \cite{Molino1988} describes Riemannian foliations in terms of TP foliations. A Riemannian foliation is called {\em complete\/} if, using Molino's theory, the corresponding TP foliation is TC. Furthermore Molino's theory describes TC foliations in terms of complete Lie foliations with dense leaves. In turn, complete Lie foliations have the following description due to Fedida \cite{Fedida1971,Fedida1973} (see also \cite[Theorem~4.1 and Lemma~4.5]{Molino1988}). Assume that $M$ is connected and $\FF$ a complete $\fg$-Lie foliation. Let $G$ be the simply connected Lie group whose Lie algebra (of left-invariant vector fields) is (isomorphic to) $\fg$. Then there is a regular covering space, $\pi:\widetilde M\to M$, a fiber bundle $D:\widetilde M\to G$ (the {\em developing\/} map) and a monomorphism $h:\Gamma:=\Aut(\pi)\equiv\pi_1L/\pi_1\widetilde L\to G$ (the {\em holonomy\/} homomorphism) such that the leaves of $\widetilde\FF:=\pi^*\FF$ are the fibers of $D$, and $D$ is $h$-equivariant with respect to the left action of $G$ on itself by left translations. As a consequence, $\pi$ restricts to diffeomorphisms between the leaves of $\widetilde\FF$ and $\FF$. The subgroup $\Hol\FF=\im h\subset G$, isomorphic to $\Gamma$, is called the {\em global holonomy group\/}. The $\widetilde\FF$-leaf through every $\tilde p\in\widetilde M$ will be denoted by $\widetilde L_{\tilde p}$. Since $D$ induces an identity $\widetilde M/\widetilde\FF\equiv G$, the $\pi$-lift and $D$-projection of vector fields define identities\footnote{Given an action, the group is added to the notation of a space of vector fields to indicate the subspace of invariant elements.}
\begin{equation}\label{overline fX(M, FF) equiv ... equiv fX(G, Hol FF)}
\olfX(M,\FF)\equiv\olfX(\widetilde M,\widetilde\FF,\Gamma)\equiv\fX(G,\Hol\FF)\;.
\end{equation}
These identities give a precise realization of $\fg\subset\olfX(M,\FF)$ as the Lie algebra of left invariant vector fields on $G$. The holonomy pseudogroup of $\FF$ is equivalent to the pseudogroup on $G$ generated by the action of $\Hol\FF$ by left translations. Thus the leaves are dense if and only if $\Hol\FF$ is dense in $G$, which means $\fg=\olfX(M,\FF)$.

\subsection{Differential forms on foliated manifolds}\label{ss: differential forms}

\subsubsection{The leafwise complex}\label{sss: leafwise complex}

Let $d_\FF\in\Diff^1(\FF;\Lambda\FF)$ be given by $(d_\FF\xi)|_L=d_L(\xi|_L)$ for every leaf $L$ and $\xi\in C^\infty(M;\Lambda\FF)$. Then $(C^\infty(M;\Lambda\FF),d_\FF)$ is a differential complex, called the {\em leafwise\/} ({\em de~Rham\/}) {\em complex\/}. This gives rise to the ({\em reduced\/}) {\em leafwise cohomology\/}\footnote{The term  {\em tangential cohomology\/} is also used.} (with complex coefficients), $H^*(\FF)=H^*(\FF;\C)$ and $\bar H^*(\FF)=\bar H^*(\FF;\C)$. Compactly supported versions may be also considered when $M$ is not compact.

Similarly, we can take coefficients in any complex $\FF$-flat vector bundle $E$ over $M$, obtaining the differential complex $(C^\infty(M;\Lambda\FF\otimes E),d_\FF)$, with $d_\FF\in\Diff^1(\FF;\Lambda\FF\otimes E)$, and the corresponding ({\em reduced\/}) {\em leafwise cohomology\/} with coefficients in $E$, $H^*(\FF;E)$ and $\bar H^*(\FF;E)$. For example, we can consider the vector bundle $E$ defined by the $\GL(n')$-principal bundle of (real) normal frames and any unitary representation of $\GL(n')$, with the $\FF$-flat structure induced by the $\FF$-flat structure of $N\FF$. A particular case is $\Lambda N\FF$, which gives rise to the differential complex $(C^\infty(M;\Lambda\FF\otimes\Lambda N\FF),d_\FF)$, and its compactly supported version. Note that
\[
\textstyle{\Lambda\FF\equiv\Lambda\FF\otimes\Lambda^0N\FF\subset\Lambda\FF\otimes\Lambda N\FF\;,}
\]
inducing an injection of topological complexes and their (reduced) cohomologies, and the same holds for the compactly supported versions. In fact, these are topological graded differential algebras with the exterior product, and the above injections are compatible with the product structures.

For any $\phi\in C^\infty(M_1,\FF_1;M_2,\FF_2)$, the morphisms $\phi_*:T\FF_1\to T\FF_2$ and $\phi_*:N\FF_1\to N\FF_2$ induce a morphism
\[
\textstyle{\phi^*:\phi^*(\Lambda\FF_2\otimes\Lambda N\FF_2)\to\Lambda\FF_1\otimes\Lambda N\FF_1}
\]
over $\id_{M_1}$, which in turn induces a continuous homomorphism of graded differential algebras, 
\begin{equation}\label{phi^*}
\phi^*:C^\infty(M_2;\Lambda\FF_2\otimes\Lambda N\FF_2)\to C^\infty(M_1;\Lambda\FF_1\otimes\Lambda N\FF_1)\;,
\end{equation}
and continuous homomorphisms between the corresponding (reduced) leafwise cohomologies. By restriction, we get the homomorphism
\[
\phi^*:C^\infty(M_2;\Lambda\FF_2)\to C^\infty(M_1;\Lambda\FF_1)\;,
\]
with analogous properties.

\subsubsection{Bigrading}\label{sss: bigrading}

Consider any splitting 
\begin{equation}\label{splitting}
TM=T\FF\oplus\bfH\cong T\FF\oplus N\FF\;,
\end{equation}
given by a transverse distribution $\bfH\subset TM$, and let $\Lambda\bfH=\bigwedge\bfH^*$. It induces a decomposition
\begin{equation}\label{Lambda M}
\Lambda M\equiv\Lambda\FF\otimes\textstyle{\Lambda\bfH}\cong\Lambda\FF\otimes\textstyle{\Lambda N\FF}\;,
\end{equation}
giving rise to the bigrading of $\Lambda M$ defined by\footnote{We have reversed the order given in \cite{AlvKordy2001} for the factors of the tensor product in the definition of $\Lambda^{u,v}M$ because the signs in some expressions become simpler. But we keep the same order for the ``transverse degree'' $u$ and the ``tangential degree'' $v$ in $\Lambda^{u,v}M$ because this is the usual order in the Leray spectral sequence of fiber bundles, generalized to foliations.}
\[
\Lambda^{u,v}M\equiv\Lambda^v\FF\otimes\textstyle{\Lambda^u\bfH}\cong\Lambda^v\FF\otimes\Lambda^uN\FF\;,
\]
and the corresponding bigrading of $C^\infty(M;\Lambda)$ with terms
\[
C^\infty(M;\Lambda^{u,v})\equiv C^\infty(M;\Lambda^v\FF\otimes\Lambda^uN\FF)\;.
\]
This bigrading depends on $\bfH$, but the spaces $\Lambda^{\ge u,\cdot}M$ and $C^\infty(M;\Lambda^{\ge u,\cdot})$ are independent of $\bfH$ (see e.g.\ \cite{Alv1989a}). In particular, every $\Lambda^{\ge u,\cdot}M/\Lambda^{\ge u+1,\cdot}M$ is independent of $\bfH$; indeed, there are canonical identities
\begin{equation}\label{bigwedge^ge u,cdot T^*M / bigwedge^ge u+1,cdot T^*M}
\Lambda^{\ge u,\cdot}M/\Lambda^{\ge u+1,\cdot}M\equiv\Lambda^{u,\cdot}M\equiv\Lambda\FF\otimes\Lambda^uN\FF\;,
\end{equation}
where only the middle bundle depends on $\bfH$. The de~Rham derivative on $C^\infty(M;\Lambda)$ decomposes into bi-homogeneous components,
\begin{equation}\label{d = d_0,1 + d_1,0 + d_2,-1}
d=d_{0,1}+d_{1,0}+d_{2,-1}\;,
\end{equation}
where the double subindex denotes the corresponding bi-degree. We have
\[
d_{0,1}\in\Diff^1(\FF;\Lambda M)\;,\quad d_{1,0}\in\Diff^1(M;\Lambda)\;,\quad d_{2,-1}\in\Diff^0(M;\Lambda)\;.
\]
Moreover\footnote{The sign of \cite[Lemma~3.4]{AlvKordy2001} is omitted here by our change in the definition of $\Lambda^{u,v}M$.}
\begin{equation}\label{d_0,1 equiv d_FF}
d_{0,1}\equiv d_\FF\;,
\end{equation}
via~\eqref{Lambda M}, and $d_{2,-1}=0$ if and only if $\bfH$ is completely integrable. Note that
\begin{equation}\label{d_0,1 = d on C^infty(M; Lambda^n',cdot)}
d_{0,1}=d:C^\infty(M;\Lambda^{n',\cdot})\to C^\infty(M;\Lambda^{n',\cdot})\;.
\end{equation}
By comparing bi-degrees in $d^2=0$, we get (see e.g.\ \cite{Alv1989a}):
\begin{equation}\label{d_0,1^2=...=0}
  d_{0,1}^2=d_{0,1}d_{1,0}+d_{1,0}d_{0,1}=0\;.
\end{equation}

For any $\phi\in C^\infty(M_1,\FF_1;M_2,\FF_2)$, we have restrictions
\[
\phi^*:C^\infty(M_2;\Lambda^{\ge u,\cdot})\to C^\infty(M_1;\Lambda^{\ge u,\cdot})
\]
of $\phi^*:C^\infty(M_2;\Lambda)\to C^\infty(M_1;\Lambda)$, which induce~\eqref{phi^*} using~\eqref{bigwedge^ge u,cdot T^*M / bigwedge^ge u+1,cdot T^*M}. Like in~\eqref{d = d_0,1 + d_1,0 + d_2,-1} and~\eqref{d_0,1 equiv d_FF},
\begin{equation}\label{phi^* = ...}
\phi^*=\phi^*_{0,0}+\phi^*_{1,-1}+\cdots:C^\infty(M_2;\Lambda)\to C^\infty(M_1;\Lambda)\;,
\end{equation}
and
\begin{equation}\label{phi^*_0,0 equiv phi^*}
\phi^*_{0,0}\equiv\phi^*
\end{equation}
 via~\eqref{Lambda M}, where the right-hand side is~\eqref{phi^*}.

For any $X\in\fX(M)$, let $\iota_X$ denote the corresponding inner product, and let $\bfV:TM\to T\FF$ and $\bfH:TM\to\bfH$ denote the projections defined by~\eqref{splitting}. By comparing bi-degrees in Cartan's formula, $\LL_X=d\iota_X+\iota_Xd$, we get a decomposition into bi-homogeneous components,
\[
\LL_X=\LL_{X,-1,1}+\LL_{X,0,0}+\LL_{X,1,-1}+\LL_{X,2,-2}\;;
\]
for instance,
\begin{equation}\label{LL_X,0,0}
\LL_{X,0,0}=d_{0,1}\iota_{\bfV X}+\iota_{\bfV X}d_{0,1}+d_{1,0}\iota_{\bfH X}+\iota_{\bfH X}d_{1,0}\;.
\end{equation}
It is easy to check that $\LL_{X,-1,1}$, $\LL_{X,1,-1}$ and $\LL_{X,2,-2}$ are of order zero, and
\begin{equation}\label{LL_X,0,0(xi wedge zeta)}
\LL_{X,0,0}(\alpha\wedge\beta)=\LL_{X,0,0}\alpha\wedge\beta+\alpha\wedge\LL_{X,0,0}\beta\;.
\end{equation}
Moreover $\LL_{X,0,0}=X$ and $\LL_{X,-1,1}=\LL_{X,1,-1}=\LL_{X,2,-2}=0$ on $C^\infty(M)$.

Assume that $X\in\fX(M,\FF)$ from now on. Then $\LL_{X,-1,1}=0$ by~\eqref{phi^* = ...}, and therefore
\begin{equation}\label{LL_X,0,0 d_{0,1} = d_{0,1} LL_X,0,0}
\LL_{X,0,0}d_{0,1}=d_{0,1}\LL_{X,0,0}\;,
\end{equation} 
as follows by comparing bi-degrees in the formula $\LL_Xd=d\LL_X$. Let $\Theta_X$ be the operator on $C^\infty(M;\Lambda\FF\otimes\Lambda N\FF)$ that corresponds to $\LL_{X,0,0}$ via~\eqref{Lambda M}. By~\eqref{LL_X,0,0(xi wedge zeta)} and~\eqref{LL_X,0,0 d_{0,1} = d_{0,1} LL_X,0,0},
\[
\Theta_X(\xi\wedge\zeta)=\Theta_X\xi\wedge\zeta+\xi\wedge\Theta_X\zeta\;,\quad\Theta_Xd_\FF=d_\FF\Theta_X\;.
\]

Let $(U,x)$ be a foliated chart of $\FF$, with $x=(x',x'')$, like in~\eqref{x = (x',x'')}. To emphasize the difference between the coordinates $x'$ and $x''$, we use the following notation on $U$ or $x(U)$. Let $x^{\prime i}=x^i$ and $\partial'_i=\partial_i$ for $i\le n'$, and $x^{\prime\prime i}=x^i$ and $\partial''_i=\partial_i$ for $i>n'$. Thus, when using $x^{\prime i}$ or $\partial'_i$, it will be understood that $i$ runs in $\{1,\dots,n'\}$, and, when using $x^{\prime\prime i}$ or $\partial''_i$, it will be understood that $i$ runs in $\{n'+1,\dots,n\}$. For multi-indices of the form $I=(i_1,\dots,i_n)\in\N_0^n$, write $\partial_I=\partial'_I\partial''_I$, where $\partial'_I=\partial_1^{i_1}\cdots\partial_{n'}^{i_{n'}}$ and $\partial''_I=\partial_{n'+1}^{i_{n'+1}}\cdots\partial_n^{i_n}$. For multi-indices of the form $J=\{j_1,\dots,j_r\}$ with $1\le j_1<\dots<j_r\le n$, let $dx^J=dx^{j_1}\wedge\dots\wedge dx^{j_r}$ be denoted by $dx^{\prime J}$ or $dx^{\prime\prime J}$ if $J$ only contains indices in $\{1,\dots,n'\}$ or $\{n'+1,\dots,n\}$, respectively. Using functions $f_I,f_{IJ}\in C^\infty(U)$, $d_\FF$ can be locally described by
\begin{equation}\label{d_FF(f_I dx^prime prime I)}
d_\FF(f_I\,dx^{\prime\prime I})=\partial''_jf_I\,dx^{\prime\prime j}\wedge dx^{\prime\prime I}\;,
\end{equation}
and~\eqref{d_0,1 equiv d_FF} means that
\begin{equation}\label{d_0,1(f_IJ dx''^I wedge dx'^J)}
d_{0,1}(f_{IJ}\,dx^{\prime\prime I}\wedge dx^{\prime J})=d_\FF(f_{IJ}\,dx^{\prime\prime I})\wedge dx^{\prime J}\;.
\end{equation}

\subsubsection{Compatibility of orientations}\label{sss: orientations}

A transverse orientation of $\FF$ can be described as a (necessarily $\nabla^\FF$-invariant) orientation of $N\FF$. It is determined by a non-vanishing real form $\omega\in C^\infty(M;\Lambda^{n'}N\FF)$; i.e., some real $\omega\in C^\infty(M;\Lambda^{n'})$ with
\[
T\FF = \{\,Y\in TM\mid\iota_Y\omega=0\,\}\;.
\]
On the other hand, an orientation of $T\FF$ is called an {\em orientation\/} of $\FF$, which can be described by a non-vanishing form $\chi\in C^\infty(M;\Lambda^{n''}\FF)\equiv C^\infty(M;\Lambda^{0,n''})$. When $\FF$ is equipped with a transverse orientation (respectively, an orientation), it is said to be {\em transversely oriented\/} (respectively, {\em oriented\/}). Given transverse and tangential orientations of $\FF$ described by forms $\omega$ and $\chi$ as above, we get an induced orientation of $M$ defined by the non-vanishing form $\chi\wedge\omega\in C^\infty(M;\Lambda^{n',n''})=C^\infty(M;\Lambda^n)$.

Suppose that moreover $M$ is a Riemannian manifold, and take $\bfH=T\FF^\perp$. Then, using~\eqref{Lambda M}, the induced Hodge star operators, $\star$ on $\Lambda M$, $\star_\FF$ on $\Lambda\FF$ and $\star_\perp$ on $\Lambda\bfH$, satisfy\footnote{The sign of this expression is different in \cite[Lemma~3.2]{AlvKordy2001} by the different choice of induced orientation of $M$, given by $\omega\wedge\chi$ in that paper.} \cite[Lemma~4.8]{AlvTond1991}, \cite[Lemma~3.2]{AlvKordy2001}
\begin{equation}\label{star equiv (-1)^u(n''-v) star_FF otimes star_perp}
\star\equiv(-1)^{u(n''-v)}{\star_\FF}\otimes{\star_\perp}:\Lambda^{u,v}M\to\Lambda^{n'-u,n''-v}M\;.
\end{equation}
If we take $\omega=\star_\perp1\in C^\infty(M;\Lambda^{n'}\bfH)\equiv C^\infty(M;\Lambda^{n',0})$ and $\chi=\star_\FF1\in C^\infty(M;\Lambda^{n''}\FF)\equiv C^\infty(M;\Lambda^{0,n''})$, then $\chi\wedge\omega=\star1\in C^\infty(M;\Lambda^n)$.

\subsubsection{Bihomogeneous components of the coderivative}

Let $g$ be a Riemannian metric on $M$. On the one hand, $g$ induces a Hermitian structure on $\Lambda\FF\otimes\Lambda N\FF$, and we can consider $\delta_\FF=d_\FF^*$ on $C^\infty(M;\Lambda\FF\otimes\Lambda N\FF)$. On the other hand, by taking formal adjoints in~\eqref{d = d_0,1 + d_1,0 + d_2,-1} with $\bfH=T\FF^\perp$, we get the decomposition into bi-homogeneous components, 
\begin{equation}\label{delta = delta_0,-1 + delta_-1,0 + delta_-2,1}
\delta=\delta_{0,-1}+\delta_{-1,0}+\delta_{-2,1}\;,
\end{equation}
where $\delta_{-i,-j}=d_{i,j}^*$. From~\eqref{d_0,1^2=...=0}, it follows that
\begin{equation}\label{delta_0,-1^2=...=0}
  \delta_{0,-1}^2=\delta_{0,-1}\delta_{-1,0}+\delta_{-1,0}\delta_{0,-1}=0\;.
\end{equation}

\begin{lem}\label{l: g bundle-like <=> delta_0,1 equiv delta_FF}
The metric $g$ is bundle-like if and only if $\delta_{0,-1}\equiv\delta_\FF$ via~\eqref{Lambda M}.
\end{lem}

\begin{proof}
By working locally, we can assume that $\FF$ is transversely oriented and oriented, and consider the induced orientation of $M$ according to Section~\ref{sss: orientations}. By~\eqref{d_0,1(f_IJ dx''^I wedge dx'^J)} and~\eqref{star equiv (-1)^u(n''-v) star_FF otimes star_perp}, and since $\star_\perp$ determines $g|_\bfH$, we get that $g$ is bundle like if and only if $d_{0,1}$ commutes with $1\otimes{\star_\perp}$, which is equivalent to $\delta_{0,-1}\equiv\delta_\FF$ by~\eqref{delta = (-1)^{nr+n+1} star d star}. 
\end{proof}

With the notation of~\eqref{d_0,1(f_IJ dx''^I wedge dx'^J)}, the equality $\delta_{0,-1}\equiv\delta_\FF$ means that
\begin{equation}\label{delta_0,-1(f_IJ dx''^I wedge dx'^J)}
\delta_{0,-1}(f_{IJ}\,dx^{\prime\prime I}\wedge dx^{\prime J})=\delta_\FF(f_{IJ}\,dx^{\prime\prime I})\wedge dx^{\prime J}\;.
\end{equation}

\section{Riemannian foliations of bounded geometry}\label{s: Riem folns of bd geometry}

With the notation of Section~\ref{s: prelims on folns}, suppose that $\FF$ is Riemannian. Let $g$ be a bundle-like metric on $M$, $\nabla$ its Levi-Civita connection and $R$ its curvature. 

The vector subbundle $T\FF^\perp\subset TM$ is called {\em horizontal\/}, giving rise to the concepts of {\em horizontal\/} vectors, vector fields and frames. Now, we take $\bfH=T\FF^\perp$ in~\eqref{splitting}, and therefore $\bfV:TM\to T\FF$ and $\bfH:TM\to \bfH$ are the orthogonal projections. The O'Neill tensors \cite{ONeill1966} of the local Riemannian submersions defining $\FF$ can be combined to produce $(1,2)$-tensors $\sT$ and $\sA$ on $M$, defined by
\begin{align*}
\sT_EF&=\bfH\nabla_{\bfV E}(\bfV F)+\bfV\nabla_{\bfV E}(\bfH F)\;,\\
\sA_EF&=\bfH\nabla_{\bfH E}(\bfV F)+\bfV\nabla_{\bfH E}(\bfH F)\;,
\end{align*}
for all $E,F\in\fX(M)$. According to \cite[Theorem~4]{ONeill1966}, if $M$ is connected, given $g$ and any $p\in M$, the foliation $\FF$ is determined by $\sT$, $\sA$ and $T_p\FF$. 

A Riemannian connection $\rnabla$ on $M$, called {\em adapted\/}, is defined by \cite{AlvTond1991}
  $$
\rnabla_EF=\bfV\nabla_E(\bfV F)+\bfH\nabla_E(\bfH F)\;,
  $$
for all $E,F\in\fX(M)$. For $V,W\in\fX(\FF)$ and $X\in C^\infty(M;\bfH)$, we have
\begin{equation}\label{nabla-rnabla}
\nabla_V-\rnabla_V=\sT_V\;,\quad\nabla_X-\rnabla_X=\sA_X\;,
\end{equation}
and \cite[Eqs.~(3.8)--(3.10)]{AlvKordyLeichtnam2014}
\begin{align}
\nabla^\FF_VW&=\rnabla_VW\;,\label{rnabla_V W}\\
\nabla^\FF_V\overline{X}&=\overline{\rnabla_VX-\sA_XV}\;,\notag\\
\bfV([X,V])&=\rnabla_XV-\sT_VX\;.\label{bV[X,V]}
\end{align}  

By~\eqref{rnabla_V W}, the $\rnabla$-geodesics that are tangent to the leaves at some point remain tangent to the leaves at every point, and they are the geodesics of the leaves. So the leaves are $\rnabla$-totally geodesic, but not necessarily $\nabla$-totally geodesic. By the second equality of~\eqref{nabla-rnabla} and \cite[Lemma~2]{ONeill1966}, $\rnabla$ and $\nabla$ have the same geodesics orthogonal to the leaves.

Given any $p\in M$, let $x':U\to\Sigma$ be a distinguished submersion so that $p\in U$. Consider the Riemannian metric on $\Sigma$ such that $x'$ is a Riemannian submersion, and let $\cnabla$ and $\cexp$ denote the corresponding Levi-Civita connection and exponential map of $\Sigma$. From \cite[Lemma~1-(3)]{ONeill1966}, it follows that $\rnabla_XY\in\fX(U,\FF|_U)$ for all horizontal $X,Y\in\fX(U,\FF|_U)$, and moreover
\begin{equation}\label{overline rnabla_XY}
\overline{\rnabla_XY}=\cnabla_{\overline X}\overline Y\;.
\end{equation}

Let $\rexp$ denote the exponential map of the geodesic spray of $\rnabla$ (see e.g.\ \cite[pp.~96--99]{Poor1981}). Observe that the exponential map of the leaves is a restriction of $\rexp$. The maps $\rexp$ and $\cexp$ restrict to diffeomorphisms of some open neighborhoods, $V$ of $0$ in $T_pM$ and $\check V$ of $0$ in $T_{x'(p)}\Sigma$, to some open neighborhoods, $O$ of $p$ in $M$ and $\check O$ of $x'(p)$ in $\Sigma$.  Moreover we can suppose that $O\subset U$, $x'_*(V)\subset\check V$ and $x'(O)\subset\check O$. By~\eqref{overline rnabla_XY},
\begin{equation}\label{z rexp = cexp z_0}
x'\,\rexp=\cexp\,x'_*
\end{equation}
on $V\cap T_p\FF^\perp$. Let $\kappa_p$ (or simply $\kappa$) be the smooth map of some neighborhood $W$ of $0$ in $T_pM$ to $M$ defined by
$$
\kappa_p(X)=\rexp_q(\rP_{\bfH X}\bfV X)\;,
$$
where $q=\rexp_p(\bfH X)$, and $\rP_{\bfH X}:T_pM\to T_qM$ denotes the $\rnabla$-parallel transport along the $\rnabla$-geodesic $t\mapsto\rexp_p(t\bfH X)$, $0\le t\le1$, which is also a $\nabla$-geodesic because it is orthogonal to the leaves. By choosing $W$ small enough, we have $W\subset V$ and $\kappa(W)\subset O$; thus $x'_*(W)\subset\cV$ and $x'\kappa(W)\subset\cO$. For $X,Y\in W$, we have $X-Y\in T_p\FF$ if and only if $\kappa(X)$ and $\kappa(Y)$ belong to the same plaque of $U$  \cite[Proposition~6.1]{AlvKordyLeichtnam2014}. We also have $x'\kappa(X)=\cexp\,x'_*(X)$ for all $X\in W\cap T_p\FF^\perp$ by~\eqref{z rexp = cexp z_0}. Furthermore $\kappa$ defines a diffeomorphism of some neighborhood of $0$ in $T_pM$ to some neighborhood of $p$ in $M$. By choosing horizontal and vertical orthonormal frames at $p$, we get identities $T_p\FF^\perp\equiv\R^{n'}$ and $T_p\FF\equiv\R^{n''}$. Then, for some open balls centered at the origin, $B'$ in $\R^{n'}$ and $B''$ in $\R^{n''}$, we can assume that $\kappa$ is a diffeomorphism of $B'\times B''$ to some open neighborhood of $p$. From now on, we use the notation $U=\kappa(B'\times B'')$ and $\kappa^{-1}=x=(x',x'')$ on $U$, like in~\eqref{x = (x',x'')}. This foliated chart $(U,x)$ is called {\em normal\/}, as well as the foliated coordinates $x$. As usual, $g_{ij}$ denotes the corresponding coefficients of the bundle-like metric, and let $(g^{ij})=(g_{ij})^{-1}$. On $U$, we have\footnote{We use the convention that repeated indices are summed.}
\begin{align}
\bfV&=g_{ik}g^{kj}\,\partial''_j\otimes dx^{\prime i}+\partial''_i\otimes dx^{\prime\prime i}\;,\label{bV}\\
\bfH&=\partial'_i\otimes dx^{\prime i}-g_{ik}g^{kj}\,\partial''_j\otimes dx^{\prime i}\;,\label{bH}
\end{align}
where $k$ runs in $\{n'+1,\dots,n\}$ \cite[Eq.~(7.2)]{AlvKordyLeichtnam2014}.

It will be said that $\FF$ has {\em positive injectivity bi-radius\/}\footnote{In \cite[Section~8]{AlvKordyLeichtnam2014}, the concept of transverse injectivity radii was introduced for a defining cocycle, and it was wrongly stated that its positivity is independent of the defining cocycle. Then some step in the proof of \cite[Theorem~8.4]{AlvKordyLeichtnam2014} does not work. This problem is clearly solved with the new concept of positive injectivity bi-radius.} if there are normal foliated coordinates $x_p:U_p\to B'\times B''$ at every $p\in M$ such that the balls $B'$ and $B''$ are independent of $p$. 

\begin{defn}[Alvarez-Kordyukov-Leichtnam {\cite[Definition~8.1]{AlvKordyLeichtnam2014}}]\label{d: bd geometry}
It is said that $\FF$ is of {\em bounded geometry\/} if it has positive injectivity bi-radius, and the functions $|\nabla^mR|$, $|\nabla^m\sT|$ and $|\nabla^m\sA|$ are uniformly bounded on $M$ for every $m\in\N_0$. 
\end{defn}

Another definition of bounded geometry for Riemannian foliations was given by Sanguiao \cite[Definition~1.7]{Sanguiao2008}. Definition~\ref{d: bd geometry} also has the following chart characterization, which is at least as strong as Sanguiao's definition \cite[Remark~8.5]{AlvKordyLeichtnam2014}.

\begin{thm}[{\'Alvarez-Kordyukov-Leichtnam \cite[Theorem~8.4]{AlvKordyLeichtnam2014}}]\label{t: foln of bd geometry}
  With the above notation, $\FF$ is of bounded geometry if and only if there is a normal foliated chart $x_p:U_p\to B'\times B''$ at every $p\in M$, such that the balls $B'$ and $B''$ are independent of $p$, and the corresponding coefficients $g_{ij}$ and $g^{ij}$, as family of smooth functions on $B'\times B''$ parametrized by $i$, $j$ and $p$, lie in a bounded subset of the Fr\'echet space $C^\infty(B'\times B'')$.
\end{thm}

In this section, assume from now on that $\FF$ is of bounded geometry. Then $M$ and the disjoint union of the leaves are of bounded geometry \cite[Remark~8.2 and Proposition~8.6]{AlvKordyLeichtnam2014}. Consider the foliated charts $x_p:U_p\to B'\times B''$ given by Theorem~\ref{t: foln of bd geometry}. The radii of the balls $B'$ and $B''$ will be denoted by $r'_0$ and $r''_0$. By the usual expression of the Christoffel symbols $\Gamma_{ij}^k$ of $\nabla$ in terms of the metric coefficients $g_{ij}$ and $g^{ij}$, and by~\eqref{bV} and~\eqref{bH}, it follows that the Christoffel symbols $\rGamma_{ij}^k$ of $\rnabla$, as family of smooth functions on $B'\times B''$ parametrized by $i$, $j$, $k$ and $p$, also lie in a bounded subset of the Fr\'echet space $C^\infty(B'\times B'')$.

\begin{prop}[{\'Alvarez-Kordyukov-Leichtnam \cite[Proposition~8.6]{AlvKordyLeichtnam2014}}]\label{p: B(p,r) subset U_p}
For some $r>0$, we have $B(p,r)\subset U_p$ for all $p\in M$. 
\end{prop} 

\begin{prop}[{\'Alvarez-Kordyukov-Leichtnam \cite[Proposition~8.7]{AlvKordyLeichtnam2014}}]
\label{p: changes of normal foliated coordinates}
For every multi-index $I$, the function $|\partial_I(x_qx_p^{-1})|$ is bounded on $x_p(U_p\cap U_q)$, uniformly on $p,q\in M$. 
\end{prop}

For $0<r'\le r'_0$ and $0<r''\le r''_0$, let $B'_{r'}$ and $B''_{r''}$ denote the balls in $\R^{n'}$ and $\R^{n''}$ centered at the origin with radii $r'$ and $r''$, respectively, and set $U_{p,r',r''}=x_p^{-1}(B'_{r'}\times B''_{r''})$.

\begin{prop}[{\'Alvarez-Kordyukov-Leichtnam \cite[Proposition~8.8]{AlvKordyLeichtnam2014}}]\label{p: p_k, N, f_k}
Let $r',r''>0$ with $2r'\le r'_0$ and $2r''\le r''_0$. Then there is a collection of points $p_k$ in $M$, and there is some $N\in\N$ such that the sets $U_{p_k,r',r''}$ cover $M$, and every intersection of $N+1$ sets $U_{p_k,2r',2r''}$ is empty. Moreover there is a partition of unity $\{f_k\}$ subordinated to the open covering $\{U_{p_k,2r',2r''}\}$, which is bounded in the Fr\'echet space $\Cinftyub(M)$.
\end{prop}

Let $y_p:V_p\to B$ be normal coordinates satisfying the statement of Theorem~\ref{t: mfd of bd geom}. The radius of $B$ is denoted by $r_0$. According to Proposition~\ref{p: B(p,r) subset U_p}, we can assume that $V_p\subset U_p$ for all $p$.

\begin{prop}\label{p: changes of normal coordinates to normal foliated coordinates}
The functions $x_py_p^{-1}$, for $p\in M$, define a bounded subset of the Fr\'echet space $C^\infty(B,\R^{n'}\times\R^{n''})$.
\end{prop}

\begin{proof}
By Theorem~\ref{t: mfd of bd geom}, the statement is equivalent to requiring that, for all $m\in\N_0$, the functions $|\nabla^mx_p|$ are bounded on $V_p$, uniformly on $p\in M$. By~\eqref{nabla-rnabla}, this in turn is equivalent to requiring that the functions $|\rnabla^mx_p|$ are bounded on $V_p$, uniformly on $p\in M$. But this follows from Theorem~\ref{t: foln of bd geometry}, since the functions $x_px_p^{-1}=\id_{B'\times B''}$ obviously define a bounded subset of the Fr\'echet space $C^\infty(B'\times B'',\R^{n'}\times\R^{n''})$.
\end{proof}

Take $0<r'<r'_0$ and $0<r''<r''_0$ such that $r'+r''<r_0$. Then $U_{p,r',r''}\subset V_p$ for all $p\in M$ by the triangle inequality. The proof of the following result is similar to the proof of Proposition~\ref{p: changes of normal coordinates to normal foliated coordinates}.

\begin{prop}\label{p: changes of normal foliated coordinates to normal coordinates}
The functions $y_px_p^{-1}$, for $p\in M$, define a bounded subset of the Fr\'echet space $C^\infty(B'_{r'}\times B''_{r''},\R^n)$.
\end{prop}

Let $E$ be the Hermitian vector bundle of bounded geometry associated to the principal $\operatorname{O}(n)$-bundle of orthonormal frames on $M$ and a unitary representation of $\operatorname{O}(n)$ (Example~\ref{ex: bundles of bd geom}~\eqref{i: associated bundle}). Since $\nabla$ on $TM$ is of bounded geometry, it follows from~\eqref{nabla-rnabla} that $\rnabla$ is also of bounded geometry. Thus we get induced connections $\nabla$ and $\rnabla$ of bounded geometry on $E$ (Example~\ref{ex: connections of bd geom}~\eqref{i: connection of the associated bundle}). By~\eqref{nabla-rnabla}, we also get that $\rnabla$ can be used instead of $\nabla$ to define equivalent versions of $\|\cdot\|_{C_{\text{\rm ub}}^m}$ and $\langle\cdot,\cdot\rangle_m$ in the spaces $C_{\text{\rm ub}}^m(M;E)$ and $H^m(M;E)$ (Sections~\ref{ss: uniform sps} and~\ref{ss: Sobolev, bd geom}). By Propositions~\ref{p: changes of normal coordinates to normal foliated coordinates} and~\ref{p: changes of normal foliated coordinates to normal coordinates}, if $B'$ and $B''$ are small enough, then we can use the coordinates $(U_p,x_p)$ instead of $(V_p,y_p)$ to define equivalent versions of $\|\cdot\|'_{C_{\text{\rm ub}}^m}$ and $\langle\cdot,\cdot\rangle'_m$. Similarly, given another bundle $F$ like $E$, we can use the coordinates $(U_p,x_p)$ instead of $(V_p,y_p)$ to describe $\Diffub^m(M;E,F)$ (Section~\ref{ss: diff ops of bd geom}) by requiring that the local coefficients form a bounded subset of the Fr\'echet space $C^\infty(B'\times B'';\C^{l'}\otimes\C^{l*})$, where $l$ and $l'$ are the ranks of $E$ and $F$.

The conditions of being leafwise differential operators and having bounded geometry are preserved by compositions, and by taking transposes and formal adjoints. Moreover
\[
\Diffub(\FF;E,F)=\Diff(\FF;E,F)\cap\Diffub(M;E,F)
\]
is a filtered $\Cinftyub(M)$-submodule of $\Diff(\FF;E,F)$. The notation $\Diffub(\FF;E)$ is used if $E=F$; this is a graded subalgebra of $\Diff(\FF;E)$. The concepts of {\em uniform leafwise ellipticity\/} for operators in $\Diff(\FF;E,F)$, and {\em uniform transverse ellipticity\/} for operators in $\Diff(M;E,F)$, can be defined like uniform ellipticity (Section~\ref{ss: diff ops of bd geom}). If $P\in\Diffub^1(\FF;E)$ is uniformly leafwise elliptic and $Q\in\Diffub^1(M;E)$ is uniformly transversely elliptic, then $H^m(M;E)$ can be described with the scalar product $\langle u,v\rangle_m=\langle(1+P^*P+Q^*Q)^mu,v\rangle$ ($m\in\R$).

The normal foliated coordinates $(U_p,x_p)$ can be used in a standard way to endow $T\FF$ with the structure of a vector bundle of bounded geometry, and let $\fX_{\text{\rm ub}}(\FF)=\Cinftyub(M;T\FF)$, which equals $\fX_{\text{\rm ub}}(M)\cap\fX(\FF)$. On the other hand, let $\fX_{\text{\rm ub}}(M,\FF)=\fX_{\text{\rm ub}}(M)\cap\fX(M,\FF)$.

\section{Operators of bounded geometry on differential forms}\label{s: ops of bd geom on diff forms}

The principal $\operatorname{O}(n)$-bundle $P$ of orthonormal frames of $M$ has a reduction $Q$ with structural group $\operatorname{O}(n')\times\operatorname{O}(n'')\subset\operatorname{O}(n)$, which consists of the frames of the form $(e',e'')$, where $e'$ and $e''$ are orthonormal frames in $\bfH$ and $T\FF$, respectively. Then $\bfH$ and $T\FF$ are associated to $Q$ and the unitary representations of $\operatorname{O}(n')\times\operatorname{O}(n'')$ on $\C^{n'}$ and $\C^{n''}$ induced by the canonical unitary representations of $\operatorname{O}(n')$ and $\operatorname{O}(n'')$. Thus $\bfH$ and $T\FF$ are of bounded geometry (Example~\ref{ex: bundles of bd geom}~\eqref{i: bundle associated to a reduction}). Moreover $\rnabla$ can be restricted to connections on $\bfH$ and $T\FF$, which are of bounded geometry because they are induced by the restriction to $Q$ of the connection on $P$ defined by $\rnabla$ (Example~\ref{ex: connections of bd geom}~\eqref{i: connection of bundle associated to a reduction}). Thus every $\Lambda^{u,v}M$ is of bounded geometry (Example~\ref{ex: bundles of bd geom}~\eqref{i: natural operations}), and the connection $\rnabla$ on $\Lambda^{u,v}M$ is of bounded geometry (Example~\ref{ex: connections of bd geom}~\eqref{i: connections induced by natural operations}).

Consider the induced connections  $\nabla$ and $\rnabla$ of bounded geometry on $\Lambda M$. By using $\rnabla$ instead of $\nabla$ in the definitions of $\|\cdot\|_{C_{\text{\rm ub}}^m}$ and $\langle\cdot,\cdot\rangle_m$, it follows that the spaces $C_{\text{\rm ub}}^m(M;\Lambda)$ and $H^m(M;\Lambda)$ inherit the bigrading of $\Lambda M$, and therefore $\Cinftyub(M;\Lambda)$ and $H^{\pm\infty}(M;\Lambda)$ have an induced bigrading: their terms of bi-degree $(u,v)$ are the uniform and Sobolev spaces for $\Lambda^{u,v}M$. In particular, all of this applies to $\Lambda\FF\equiv\Lambda^{0,\cdot}M$.

\begin{lem}\label{l: the canonical projection to bigwedge^u,vT^*M is of bd geometry}
The canonical projection of $\Lambda M$ to every $\Lambda^{u,v}M$ is of bounded geometry for all $u$ and $v$.
\end{lem}

\begin{proof}
This follows from~\eqref{bV},~\eqref{bH} and Theorem~\ref{t: foln of bd geometry}.
\end{proof}

Using the decompositions~\eqref{d = d_0,1 + d_1,0 + d_2,-1} and~\eqref{delta = delta_0,-1 + delta_-1,0 + delta_-2,1}, let
\begin{gather*}
D_0=d_{0,1}+\delta_{0,-1}\;,\quad D_\perp=d_{1,0}+\delta_{-1,0}\;,\\
\Delta_0=D_0^2=d_{0,1}\delta_{0,-1}+\delta_{0,-1}d_{0,1}\;.
\end{gather*}
Note that $D_0\in\Diff^1(\FF;\Lambda M)$, $D_\perp\in\Diff^1(M;\Lambda)$ and $\Delta_0\in\Diff^2(\FF;\Lambda M)$.

\begin{cor}\label{c: d_i,j, ... are of bd geom}
  The differential operators $d_{i,j}$, $\delta_{-i,-j}$, $D_0$, $D_\perp$ and $\Delta_0$ are of bounded geometry.
\end{cor}

\begin{proof}
This follows from Lemma~\ref{l: the canonical projection to bigwedge^u,vT^*M is of bd geometry} since $d$ is of bounded geometry, and this property is preserved by taking formal adjoints and compositions.
\end{proof}

It is elementary that
\begin{alignat}{2}
\Fsigma(d_{0,1})(p,\xi)&=i{\xi\wedge}\;,&\quad\sigma(d_{-1,0})(p,\zeta)&=i{\zeta\wedge}\;,\label{Fsigma(d_0,1), sigma(d_-1,0)}\\
\Fsigma(\delta_{0,-1})(p,\xi)&=i{\xi\lrcorner}\;,&\quad\sigma(\delta_{0,-1})(p,\zeta)&=i{\zeta\lrcorner}\;,\notag
\end{alignat}
for all $p\in M$, $\xi\in T^*_p\FF$ and $\zeta\in N^*_p\FF$. So
\begin{gather}
\Fsigma(D_0)(p,\xi)=i({\xi\wedge}+{\xi\lrcorner})\;,\quad\Fsigma(\Delta_0)(p,\xi)=|\xi|^2\;,\label{Fsigma(D_0)}\\
\sigma(D_\perp)(p,\zeta)=i({\zeta\wedge}+{\zeta\lrcorner})\;.\notag
\end{gather}
Thus we get the following.

\begin{prop}\label{p: D_0 is uniformly leafwise elliptic}
$D_0$ and $\Delta_0$ are uniformly leafwise elliptic, and $D_\perp$ is uniformly transversely elliptic.
\end{prop}

Let us extend the arguments of \cite[Section~3]{AlvKordy2001} to open manifolds using bounded geometry. The expression~\eqref{bV[X,V]} defines a differential operator
\[
\Theta:\fX(\FF)\to C^\infty(M;\bfH^*\otimes T\FF)\;,\quad\Theta_XV=\bfV([X,V])\;,
\]
for $X\in C^\infty(M;\bfH)$ and $V\in\fX(\FF)$. It induces a differential operator
\begin{gather*}
\Theta:C^\infty(M;\Lambda\FF)\to C^\infty(M;\bfH^*\otimes\Lambda\FF)\;,\\
(\Theta_X\alpha)(V_1,\dots,V_r)=X\alpha(V_1,\dots,V_r)-\sum_{j=1}^r\alpha(V_1,\dots,\Theta_XV_j,\dots,V_r)\;.
\end{gather*}
 for $X\in C^\infty(M;\bfH)$, $\alpha\in C^\infty(M;\Lambda^r\FF)$ and $V_j\in\fX(\FF)$. If $X\in\fX(M,\FF)\cap C^\infty(M;\bfH)$, then this expression agrees with the operator $\Theta_X$ of Section~\ref{ss: differential forms}. According to \cite[Lemma~3.3]{AlvKordy2001}, a zero order differential operator
\[
\Xi:C^\infty(M;\Lambda\FF)\to C^\infty(M;\bfH^*\otimes\Lambda\FF)
\]
is locally defined by
\[
\Xi_X=(-1)^{(n''-v)v}[\Theta_X,\star_\FF]\star_\FF
\]
on $C^\infty(M;\Lambda^v\FF)$ for any $X\in C^\infty(M;\bfH)$, where $\star_\FF$ is the local leafwise star operator determined by $g$ and any local orientation of $\FF$. This $\Xi$ can be considered as a vector bundle morphism $\Lambda\FF\to \bfH^*\otimes\Lambda\FF$. By tensoring $\Xi$ with the identity morphism on $\Lambda\bfH$, we get a vector bundle morphism $\Lambda M\to \bfH^*\otimes\Lambda M$ according to~\eqref{Lambda M}, which is also denoted by $\Xi$. On any normal foliated chart $(U,x)$, let $K$ be the endomorphism of $\Lambda U$ given by
\[
K={dx^{\prime i}\wedge}\,\Xi_{\bfH\partial'_i}\;.
\]
This local definition gives rise to a global endomorphism $K$ of $\Lambda M$.

\begin{prop}\label{p: K is of bd geometry}
$K$ is of bounded geometry.
\end{prop}

\begin{proof}
Take a normal foliated chart $(U,x)$. By~\eqref{bH},
\[
\Theta_{\bfH\partial'_i}\partial''_b
=\bfV([\bfH\partial'_i,\partial''_b])
=\bfV([\partial'_i-g_{ik}g^{kj}\partial''_j,\partial''_b])
=\partial''_b(g_{ik}g^{kj})\partial''_j\;,
\]
where $k$ runs in $\{n'+1,\dots,n\}$. Hence
\begin{align*}
(\Theta_{\bfH\partial'_i}dx^{\prime\prime a})(\partial''_b)
&=\bfH\partial'_i(dx^{\prime\prime a}(\partial''_b))
-dx^{\prime\prime a}(\Theta_{\bfH\partial'_i}\partial''_b)\\
&=-\partial''_b(g_{ik}g^{kj})\,dx^{\prime\prime a}(\partial''_j)
=-\partial''_b(g_{ik}g^{k\alpha})\;,
\end{align*}
giving
\[
\Theta_{\bfH\partial'_i}dx^{\prime\prime a}
=-\partial''_b(g_{ik}g^{ka})\,dx^{\prime\prime b}\;.
\]
It follows that $\Theta_{\bfH\partial'_i}dx^{\prime\prime I}=f_{iIK}\,dx^{\prime\prime K}$, where the functions $f_{iIK}$ are universal polynomial expressions of the functions $g_{ab}$ and $g^{ab}$, and their partial derivatives. On the other hand, for any choice of an orientation of $\FF$ on $U$, we have $\star_\FF dx^{\prime\prime I}=h_{IK}\,dx^{\prime\prime K}$, where the functions $h_{IK}$ are universal expressions of the functions $g_{ab}$ and $g^{ab}$. So
\begin{align*}
\Xi_{\bfH\partial'_i}dx^{\prime\prime I}
&=(-1)^{(n''-v)v}[\Theta_{\bfH\partial'_i},\star_\FF]\star_\FF dx^{\prime\prime I}\\
&=\Theta_{\bfH\partial'_i}dx^{\prime\prime I}
-(-1)^{(n''-v)v}\star_\FF\Theta_{\bfH\partial'_i}\star_\FF dx^{\prime\prime I}\\
&=f_{iIK}\,dx^{\prime\prime K}
-(-1)^{(n''-v)v}\star_\FF\Theta_{\bfH\partial'_i}(h_{IL}\,dx^{\prime\prime A})\\
&=f_{iIK}\,dx^{\prime\prime K}
-(-1)^{(n''-v)v}\star_\FF\big((\partial'_i-g_{ik}g^{kj}\partial''_j)(h_{IA})\,dx^{\prime\prime A}\\
&\phantom{=\text{}}\text{}+h_{IA}f_{iAB}\,dx^{\prime\prime B}\big)\\
&=\big(f_{iIK}-(-1)^{(n''-v)v}\big((\partial'_ih_{IA}-g_{ik}g^{kj}\partial''_jh_{IA})h_{AK}\\
&\phantom{=\text{}}\text{}+h_{IA}f_{iAB}h_{BK}\big)\big)\,dx^{\prime\prime K}\;.\qedhere
\end{align*}
\end{proof}

Like in the case of compact manifolds \cite[Proposition~3.1]{AlvKordy2001}, using the local expression $\delta_\FF=(-1)^v\star_\FF d_\FF\,\star_\FF$ on $C^\infty(U;\Lambda^v\FF)$, and~\eqref{d_0,1^2=...=0},~\eqref{d_0,1(f_IJ dx''^I wedge dx'^J)},~\eqref{delta_0,-1^2=...=0} and~\eqref{delta_0,-1(f_IJ dx''^I wedge dx'^J)}, we get
\begin{equation}\label{D_perp D_0 + D_0 D_perp}
D_\perp D_0+D_0D_\perp=KD_0+D_0K\;.
\end{equation}

\section{Leafwise Hodge decomposition}\label{s: Leafwise Hodge}

With the notation of Section~\ref{s: Riem folns of bd geometry}, since $M$ is complete, by~\eqref{Fsigma(D_0)} and the commutativity of~\eqref{CD, Fsigma_m, sigma_m}, given any $\alpha\in\Cinftyc(M;\Lambda)$, the hyperbolic equation
\begin{equation}\label{leafwise wave}
  \partial_t\alpha_t=iD_0\alpha_t\;,\quad\alpha_0=\alpha\;,
\end{equation}
has a unique solution $\alpha_t\in\Cinftyc(M;\Lambda)$ depending smoothly on $t\in\R$ \cite[Theorem~1.3]{Chernoff1973}. The solutions of~\eqref{leafwise wave} defined on any open subset of $M$ and for $t$ in any interval containing zero satisfy (see \cite[Proposition~1.2]{Roe1987})
\begin{equation}\label{leafwise unit propagation speed}
\supp\alpha_t\subset\Pen_\FF(\supp\alpha,|t|)\;.
\end{equation}
This can be proved like~\eqref{unit propagation speed}, or it also follows from~\eqref{unit propagation speed} by taking restrictions to the leaves.

The operators $D_0$ and $\Delta_0$, with domain $\Cinftyc(M;\Lambda)$, are essentially self-adjoint in $L^2(M;\Lambda)$ \cite[Theorem~2.2]{Chernoff1973}, and their self-adjoint extensions are also denoted by $D_0$ and $\Delta_0$. Using the functional calculus of $D_0$ given by the spectral theorem, we get a unitary operator $e^{itD_0}$ and a bounded self-adjoint operator $e^{-t\Delta_0}$ on $L^2(M;\Lambda)$ with $\|e^{-t\Delta_0}\|\le1$. The solution of~\eqref{leafwise wave} is given by $\alpha_t=e^{itD_0}\alpha$. Let $\Pi_0$ (or $e^{-\infty\Delta_0}$) denote the orthogonal projection of $L^2(M;\Lambda)$ to the kernel of $\Delta_0$ in $L^2(M;\Lambda)$.

\begin{prop}\label{p: |e^itD_0 alpha|_m}
For every $m\in\N_0$, there is some $C_m\ge0$ such that, for all $\alpha\in\Cinftyc(M;\Lambda)$ and $t\in\R$,
\[
\|e^{itD_0}\alpha\|_m\le e^{C_m|t|}\|\alpha\|_m\;.
\]
\end{prop}

\begin{proof}
We adapt arguments from \cite[Section~IV.2]{Taylor1981}. The case where $M$ is compact is stated in \cite[Proposition~1.4]{Roe1987} with more generality.

By Proposition~\ref{p: D_0 is uniformly leafwise elliptic}, we can assume that, for all $\alpha\in\Cinftyc(M;\Lambda)$,
\[
\|\alpha\|_m=\|\alpha\|+\|D_0^m\alpha\|+\|D_\perp^m\alpha\|\;.
\]
Writing $\alpha_t=e^{itD_0}\alpha$, we have
\begin{align*}
\frac{d}{dt}\|D_0^m\alpha_t\|^2&=\langle iD_0^{m+1}\alpha_t,D_0^m\alpha_t\rangle+\langle D_0^m\alpha_t,iD_0^{m+1}\alpha_t\rangle=0\;,\\
\frac{d}{dt}\|D_\perp^m\alpha_t\|^2&=\langle iD_\perp^mD_0\alpha_t,D_\perp^m\alpha_t\rangle+\langle D_\perp^m\alpha_t,iD_\perp^mD_0\alpha_t\rangle\\
&=i\langle[\Delta_\perp^m,D_0]\alpha_t,\alpha_t\rangle\;.
\end{align*}
But, by~\eqref{D_perp D_0 + D_0 D_perp},
\begin{align*}
[\Delta_\perp^m,D_0]&=\sum_{j=0}^{m-1}\Delta_\perp^{m-1-j}[D_\perp,D_\perp D_0+D_0D_\perp]\Delta_\perp^j\\
&=\sum_{j=0}^{m-1}\Delta_\perp^{m-1-j}[D_\perp,KD_0+D_0K]\Delta_\perp^j\;.
\end{align*}
This expression can be written as a sum of $4m$ terms, $\sum_lP_lQ_l$, where, up to sign, $P_l$ and $Q_l$ are operators of the one of the following forms: $D_\perp^aKD_0D_\perp^b$, $D_\perp^aD_0KD_\perp^b$ or $D_\perp^m$, for $a,b\in\N_0$ with $a+b+1=m$. In particular, $P_l,Q_l\in \Diffub^m(M;\Lambda)$ by Corollary~\ref{c: d_i,j, ... are of bd geom} and Lemma~\ref{p: K is of bd geometry}. Hence there is some $C_m>0$, independent of $\alpha$, such that
\[
\frac{d}{dt}\|D_\perp^m\alpha_t\|^2\le\sum_l|\langle Q_l\alpha_t,P_l^*\alpha\rangle|\le\sum_l\|Q_l\alpha_t\|\cdot\|P_l^*\alpha\|\le C_m\|\alpha_t\|_m^2\;.
\]
Therefore
\[
\frac{d}{dt}\|\alpha_t\|_m^2\le C_m\|\alpha_t\|_m^2\;,
\]
and the result follows by using Gronwall's inequality.
\end{proof}

Recall also that the Schwartz space $\SS=\SS(\R)$ is the Fr\'echet space of functions $\psi\in C^\infty(\R)$ such that $\psi^{(m)}\in\RR$ for all $m\in\N_0$, with the semi-norms defined by applying the semi-norms of $\RR$ to derivatives of arbitrary order.  Let $\AA$ denote the Fr\'echet algebra and $\C[z]$-module of functions $\psi:\R\to\C$ that can be extended to entire functions on $\C$ such that, for every compact $K\subset\R$, the set $\{\,x\mapsto \psi(x+iy)\mid y\in K\,\}$ is bounded in $\SS$ \cite[Section~4]{Roe1987}. It contains all functions with compactly supported smooth Fourier transform, as well as the Gaussian $x\mapsto e^{-x^2}$. Furthermore, if $\psi\in\AA$ and $u>0$, then $\psi_u\in\AA$, where $\psi_u(x)=\psi(ux)$. By the Paley-Wiener theorem, for every $\psi\in\AA$ and $c>0$, there is some $A_c>0$ such that, for all $\xi\in\R$,
\begin{equation}\label{|hat psi(xi)| le A_c e^-c |xi|}
\big|\hat\psi(\xi)\big|\le A_ce^{-c|\xi|}\;.
\end{equation}
Semi-norms on $\AA$, $\|\cdot\|_{\AA,C,r}$ ($C>0$ and $r\in\N_0$), can be defined by
\[
\|\psi\|_{\AA,C,r}=\max_{j+k\le r} \int_\infty^\infty|\xi^j\partial^k_{\xi}\hat{\psi}(\xi)|\,e^{C|\xi|}\,d\xi\;.
\]

\begin{prop}\label{p: functional calculus}
  The functional calculus map, $\psi\mapsto \psi(D_0)$, restricts to a continuous homomorphism $\AA\to\End(H^\infty(M;\Lambda))$ of $\C[z]$-modules and algebras.
\end{prop}

\begin{proof}
This follows like in the case where $M$ is compact \cite[Proposition~4.1]{Roe1987}. Precisely, for every $\psi\in\AA$, it follows from the inverse Fourier transform that
\begin{equation}\label{psi(D_0)}
\psi(D_0)=\frac{1}{2\pi}\int_{-\infty}^\infty\hat\psi(\xi)e^{i\xi D_0}\,d\xi\;.
\end{equation}
So, by Proposition~\ref{p: |e^itD_0 alpha|_m} and~\eqref{|hat psi(xi)| le A_c e^-c |xi|}, $\psi(D_0)$ defines an endomorphism of every $H^m(M;\Lambda)$ with
\begin{align}
\|\psi(D_0)\|_m&\le\frac{1}{2\pi}\int_{-\infty}^\infty|\hat\psi(\xi)|\,e^{C_m|\xi|}\,d\xi\label{|psi(D_0)|_m}\\
&\le\frac{A_k}{2\pi}\int_{-\infty}^\infty e^{(C_m-k)|\xi|}\,d\xi\;,\notag
\end{align}
which is finite for $k>C_m$.
\end{proof}

According to Proposition~\ref{p: functional calculus}, the operator $e^{-t\Delta_0}$ ($t>0$) restricts to a continuous endomorphism of $H^\infty(M;\Lambda)$. As pointed out in \cite{Sanguiao2008}, by using~\eqref{D_perp D_0 + D_0 D_perp}, Corollary~\ref{c: d_i,j, ... are of bd geom}, and Propositions~\ref{p: D_0 is uniformly leafwise elliptic},~\ref{p: K is of bd geometry} and~\ref{p: functional calculus}, the arguments of the proof of \cite[Theorem~A]{AlvKordy2001} can be adapted to show the following result, where $\Delta_0$ is considered on $H^\infty(M;\Lambda)$.

\begin{thm}[Sanguiao \cite{Sanguiao2008}]\label{t: leafwise Hodge decomposition}
There is a topological direct sum decomposition, 
\begin{equation}\label{leafwise Hodge decomposition}
H^\infty(M;\Lambda)=\ker\Delta_0\oplus\overline{\im d_{0,1}}\oplus\overline{\im\delta_{0,-1}}\;.
\end{equation}
Moreover $(t,\alpha)\mapsto e^{-t\Delta_0}\alpha$ defines a continuous map
\[
[0,\infty]\times H^\infty(M;\Lambda)\to H^\infty(M;\Lambda)\;.
\]
\end{thm}

By Corollary~\ref{c: d_i,j, ... are of bd geom}, $(H^\infty(M;\Lambda),d_{0,1})$ is a topological complex. The terms of the direct sum decomposition~\eqref{leafwise Hodge decomposition} are orthogonal in $L^2(M;\Lambda)$. Thus $\Pi_0$ has a restriction $H^\infty(M;\Lambda)\to\ker\Delta_0$, which induces the isomorphism stated in the following corollary. Its inverse is induced by the inclusion map $\ker\Delta_0\hookrightarrow H^\infty(M;\Lambda)$.

\begin{cor}[Sanguiao \cite{Sanguiao2008}]\label{c: leafwise Hodge iso}
  As topological vector spaces,
\[
\bar H^*(H^\infty(M;\Lambda),d_{0,1})\cong\ker\Delta_0\;.
\]
\end{cor}

By~\eqref{Lambda M} and~\eqref{d_0,1 equiv d_FF}, we can consider $(H^\infty(M;\Lambda\FF),d_\FF)$ as a topological subcomplex of $(H^\infty(M;\Lambda),d_{0,1})$, and the notation $H^*H^\infty(\FF)$ and $\bar H^*H^\infty(\FF)$ is used for its (reduced) cohomology. By  Lemma~\ref{l: g bundle-like <=> delta_0,1 equiv delta_FF}, $\delta_\FF$ on $H^\infty(M;\Lambda\FF)$ is also given by $\delta_{0,-1}$. Thus we get the operators $D_\FF=d_\FF+\delta_\FF$ and $\Delta_\FF=D_\FF^2=\delta_\FF d_\FF+d_\FF\delta_\FF$ on $H^\infty(M;\Lambda\FF)$, which are essentially self-adjoint in $L^2(M;\Lambda\FF)$. Then Propositions~\ref{p: |e^itD_0 alpha|_m} and~\ref{p: functional calculus}, Theorem~\ref{t: leafwise Hodge decomposition} and Corollary~\ref{c: leafwise Hodge iso} have obvious versions for $d_\FF$, $\delta_\FF$, $D_\FF$ and $\Delta_\FF$; in particular, we get the following.

\begin{thm}[Sanguiao \cite{Sanguiao2008}]\label{t: leafwise Hodge decomposition for d_FF}
Let $\FF$ be a Riemannian foliation of bounded geometry on a Riemannian manifold $M$ with a bundle-like metric. Then there is a topological direct sum decomposition, 
\begin{equation}\label{leafwise Hodge decomposition for d_FF}
H^\infty(M;\Lambda\FF)=\ker\Delta_\FF\oplus\overline{\im d_\FF}\oplus\overline{\im\delta_\FF}\;.
\end{equation}
Moreover $(t,\alpha)\mapsto e^{-t\Delta_\FF}\alpha$ defines a continuous map
\[
[0,\infty]\times H^\infty(M;\Lambda\FF)\to H^\infty(M;\Lambda\FF)\;.
\]
\end{thm}

\begin{cor}[Sanguiao \cite{Sanguiao2008}]\label{c: leafwise Hodge iso for d_FF}
$\bar H^*H^\infty(\FF)\cong\ker\Delta_\FF$ as topological vector spaces.
\end{cor}

Theorem~\ref{t: leafwise Hodge decomposition} and Corollary~\ref{c: leafwise Hodge iso} can be considered as versions of Theorem~\ref{t: leafwise Hodge decomposition for d_FF} and Corollary~\ref{c: leafwise Hodge iso for d_FF} with coefficients in $\Lambda N\FF$. We could also take leafwise differential forms with coefficients in other Hermitian vector bundles associated to $N\FF$, like the bundles of transverse densities or transverse symmetric tensors\footnote{However this is not true for any Hermitian vector bundle with a flat Riemannian $\FF$-partial connection, contrary to what was wrongly asserted in \cite[Corollary~C]{AlvKordy2001} when $M$ is compact. A counterexample was provided to the first two authors by S.~Goette.}. 

Compare Theorem~\ref{t: leafwise Hodge decomposition} and Corollary~\ref{c: leafwise Hodge iso for d_FF} with~\eqref{Hodge dec} and~\eqref{Hodge iso} (the case of an elliptic complex on a closed manifold).

\section{Foliated maps of bounded geometry}\label{ss: foliated maps of bd geom}

For $a=1,2$, let $\FF_a$ be a Riemannian foliation of bounded geometry on a manifold $M_a$ with a bundle-like metric. Set $n'_a=\codim\FF_a$ and $n''_a=\dim\FF_a$. Consider a normal foliated chart $x_{a,p}:U_{a,p}\to B'_a\times B''_a$ at every $p\in M_a$ satisfying the conditions of Theorem~\ref{t: foln of bd geometry}. Let $r'_a$ and $r''_a$ denote the radii of $B'_a$ and $B''_a$, respectively. For $0<r'<r'_a$ and $0<r''<r''_a$, let $U_{a,p,r',r''}=x_{a,p}^{-1}(B'_{a,r'}\times B''_{a,r''})$, where $B'_{a,r'}\subset\R^{n'_a}$ and $B''_{a,r''}\subset\R^{n''_a}$ denote the balls centered at the origin with respective radii $r'$ and $r''$. Like in the cases of $C_{\text{\rm ub}}^m(M;E)$, $H^m(M;E)$ and $\Diffub^m(M;E,F)$ (Section~\ref{s: Riem folns of bd geometry}), in the definition of bounded geometry for maps $M_1\to M_2$ (Section~\ref{ss: maps of bd geom}), we can replace the charts $(V_{1,p},y_{1,p})$ and $(V_{2,\phi(p)},y_{2,\phi(p)})$ with the charts $(U_{1,p},x_{1,p})$ and $(U_{2,\phi(p)},x_{2,\phi(p)})$, and we can replace the sets $B_1(p,r)$ with the sets $U_{1,p,r',r''}$. Let
\[
\Cinftyub(M_1,\FF_1;M_2,\FF_2)=C^\infty(M_1,\FF_1;M_2,\FF_2)\cap \Cinftyub(M_1,M_2)\;.
\]

For $m\in\N_0$ and $\phi\in \Cinftyub(M_1,\FF_1;M_2,\FF_2)$, using the version of $\|\cdot\|'_{C^m_{\text{\rm ub}}}$ defined with the charts $(U_p,x_p)$, it follows that $\phi^*$ induces a bounded homomorphism
\[
\phi^*:C_{\text{\rm ub}}^m(M_2;\Lambda\FF_2)
\to C_{\text{\rm ub}}^m(M_1;\Lambda\FF_1)\;,
\]
obtaining a continuous homomorphism
\[
\phi^*:\Cinftyub(M_2;\Lambda\FF_2)\to \Cinftyub(M_1;\Lambda\FF_1)\;.
\]
These homomorphisms are induced by~\eqref{phi^* on C_ub^m} and~\eqref{phi^* on C_ub^infty} via~\eqref{Lambda M}. Similarly, if $\phi$ is also uniformly metrically proper, then $\phi^*$ induces a bounded homomorphism
\[
\phi^*:H^m(M_2;\Lambda\FF_2)\to H^m(M_1;\Lambda\FF_1)
\]
for all $m$, and therefore it induces a continuous homomorphism
\[
\phi^*:H^{\pm\infty}(M_2;\Lambda\FF_2)\to H^{\pm\infty}(M_1;\Lambda\FF_1)\;.
\]
By~\eqref{phi^*_0,0 equiv phi^*}, these homomorphisms are induced by~\eqref{phi^* on H^m} and~\eqref{phi^* on H^pm infty} via~\eqref{Lambda M}.

Now, let $\phi=\{\phi^t\}$ be a foliated flow on $(M,\FF)$ of $\R$-local bounded geometry, and let $Z\in\fX_{\text{\rm ub}}(M,\FF)$ be its infinitesimal generator (Proposition~\ref{p: X is C^infty-uniformly bd <=> phi^t is of R-local bd geom}). Every $\phi^t$ is uniformly metrically proper because $\phi^{\pm t}$ is of bounded geometry (Section~\ref{ss: maps of bd geom}). Thus $\phi^{t*}$ induces continuous homomorphisms
\[
\Cinftyub(M;\Lambda\FF)\to \Cinftyub(M;\Lambda\FF)\;,\quad
H^{\pm\infty}(M;\Lambda\FF)\to H^{\pm\infty}(M;\Lambda\FF)\;.
\]

\begin{prop}\label{p: Theta_Z is of bd geometry}
$\Theta_Z\in\Diffub^1(M;\Lambda\FF)$.
\end{prop}

\begin{proof}
 Since $d\in\Diffub^1(M;\Lambda\FF)$ and $Z\in\fX_{\text{\rm ub}}(M,\FF)$, we get that $\LL_Z=d\iota_Z+\iota_Zd$ is of bounded geometry. So $\Theta_Z\equiv\LL_{Z,0,0}$ is of bounded geometry by Lemma~\ref{l: the canonical projection to bigwedge^u,vT^*M is of bd geometry} and using the identity~\eqref{Lambda M}.
\end{proof}

By~\eqref{LL_X,0,0} and~\eqref{Fsigma(d_0,1), sigma(d_-1,0)},
\[
\sigma(\Theta_Z)(p,\zeta)=i\zeta(Z)
\]
for all $p\in M$ and $\zeta\in N^*_pM$, obtaining the following.

\begin{prop}\label{p: Theta_Z is unif transv elliptic}
$\Theta_Z$ is uniformly transversely elliptic if $\inf_M|\overline Z|>0$.
\end{prop}

\section{A class of smoothing operators}\label{s: smoothing operators}

Suppose that $\FF$ is of codimension one\footnote{The higher dimensional case could be treated like in \cite{AlvKordy2008a}, but we only consider codimension one here for the sake of simplicity.}. Assume also that $M$ is equipped with a bundle-like metric $g$ so that $\FF$ is of bounded geometry. Let $\phi=\{\phi^t\}$ be a foliated flow of $\R$-local bounded geometry, whose infinitesimal generator is $Z\in\fX_{\text{\rm ub}}(M,\FF)$ (Proposition~\ref{p: X is C^infty-uniformly bd <=> phi^t is of R-local bd geom}). Suppose that $\inf_M|\overline Z|>0$; in particular, the orbits of $\phi$ are transverse to the leaves. Moreover let $A=\{\,A_t\mid t\in\R\,\}\subset\Diff^m(\FF;\Lambda\FF)$ be a smooth $\R$-compactly supported family of $\R$-local bounded geometry. For every $\psi\in\AA$, the operator
\begin{equation}\label{P}
P=\int_\R\phi^{t*}A_t\,dt\,\psi(D_\FF)
\end{equation}
on $H^{-\infty}(M;\Lambda\FF)$ is defined by the version of Proposition~\ref{p: functional calculus} for $D_\FF$. The subindex ``$\psi$'' may be added to the notation of $P$ if needed, or the subindex ``$u$'' in the case of functions $\psi_u\in\AA$ depending on a parameter $u$.

\begin{prop}\label{p: P}
$P_\psi$ is a smoothing operator, and the linear map
\[
\AA\to L(\textstyle{H^{-\infty}(M;\Lambda\FF),H^\infty(M;\Lambda\FF)})\;,\quad\psi\mapsto P_\psi\;,
\]
is continuous.
\end{prop}

\begin{proof}
According to the proof of Proposition~\ref{p: functional calculus}, $\psi(D_\FF)$ defines a bounded operator on every $H^m(M;\Lambda\FF)$. Since moreover $\phi$ and $A$ are of $\R$-local bounded geometry, and $A$ is $\R$-compactly supported, it follows that $P$ also defines a bounded operator on every $H^m(M;\Lambda\FF)$.

Since $\Theta_Z\in\Diff^1_{\text{\rm ub}}(M;\Lambda\FF)$ is uniformly transversely elliptic (Propositions~\ref{p: Theta_Z is of bd geometry} and~\ref{p: Theta_Z is unif transv elliptic}) and $D_\FF\in\Diff^1_{\text{\rm ub}}(\FF;\Lambda\FF)$ is uniformly leafwise elliptic (Corollary~\ref{c: d_i,j, ... are of bd geom} and Proposition~\ref{p: D_0 is uniformly leafwise elliptic}), to get that $P$ is smoothing, it suffices to prove that $\Theta_Z^NP$ and $D_\FF^NP$ are of the form~\eqref{P} for all $N\in\N_0$. In turn, this follows by showing that $\Theta_ZP$ and $QP$ are of the form~\eqref{P} for any $Q\in\Diffub(\FF;\Lambda\FF)$.

We have
$$
QP=\int_\R\phi^{t*}B_t\,dt\,\psi(D_\FF)\;,
$$
where $B_t=\phi^{-t*}Q\phi^{t*}A_t$. Since $\phi^t$ is a foliated map, this defines a smooth family $B=\{\,B_t\mid t\in\R\,\}\subset\Diff(\FF;\Lambda\FF)$. Moreover $B$ is $\R$-compactly supported and of $\R$-local bounded geometry because $A$ is $\R$-compactly supported, and $\phi$, $Q$ and $A$ are of $\R$-local bounded geometry. Thus $QP$ is of the form~\eqref{P}.

Let $C=\{\,C_t\mid t\in\R\,\}\subset\Diff(\FF;\Lambda\FF)$ be the smooth family given by $C_t=\frac{d}{ds}A_{t-s}|_{s=0}$. Note that $C$ is of $\R$-local bounded geometry because the family $A$ is of $\R$-local bounded geometry and $\R$-compactly supported. Like in the proof of \cite[Proposition~6.1]{AlvKordy2008a}, we get
\begin{align*}
\Theta_ZP&=\left.\frac{d}{ds}\int_\R\phi^{t+s*}A_t\,dt\right|_{s=0}\,\psi(D_\FF)\\
&=\left.\frac{d}{ds}\int_\R\phi^{r*}A_{r-s}\,dr\right|_{s=0}\,\psi(D_\FF)=\int_\R\phi^{t*}C_t\,dt\,\psi(D_\FF)\;,
\end{align*}
which is of the form~\eqref{P}.

By~\eqref{|psi(D_0)|_m}, for any $N\in\N_0$ and $\psi\in\AA$, the operator $(1+\Delta_{\FF})^N\psi(D_\FF)$ extends to a bounded operator on every $H^m(M;\Lambda\FF)$ with
\begin{equation}\label{|(1+Delta_FF)^N psi(D_FF)|_m}
\|(1+\Delta_\FF)^N\psi(D_\FF)\|_m
\le\frac{1}{2\pi}\int_{-\infty}^\infty|(1-\partial^2_{\xi})^N\hat{\psi}(\xi)|\,e^{C_m|\xi|}\,d\xi\;.
\end{equation}
Hence, by the above argument, it can be easily seen that, for integers $m\le m'$, there are some $C,C'>0$ and $N\in\N_0$ such that
\begin{equation}\label{|P|_m,m'}
\|P\|_{m,m'}\le C'\int
|(\id-\partial^2_{\xi})^N\hat{\psi}(\xi)|\,e^{C|\xi|}\,d\xi
\le C'\|\psi\|_{\AA,C,2N}\;.
\end{equation}
 Here, $C$ depends on $m$ and $m'$, and $C'$ depends on $m$, $m'$ and $A$. Then the mapping $\psi\mapsto P_\psi$ of the statement is continuous.
\end{proof}

\begin{cor}\label{c: K_P}
The linear map
\[
\AA\to\Cinftyub(M^2;\Lambda\FF\boxtimes(\Lambda\FF^*\otimes\Omega M))\;,\quad\psi\mapsto K_{P_\psi}\;,
\]
is continuous.
\end{cor}

\begin{proof}
This follows from Propositions~\ref{p: P} and~\ref{p: Schwartz kernel with bd geometry}.
\end{proof}

Now, consider the particular case where $\psi_u(x)=e^{-ux^2}$, and the corresponding operators $P_u$ ($u>0$) on $H^\infty(M;\Lambda\FF)$. Let also
\[
P_\infty=\int_\R\phi^{t*}A_t\,dt\,\Pi_\FF
\]
on $H^\infty(M;\Lambda\FF)$, where $\Pi_\FF$ is the orthogonal projection to $\ker\Delta_\FF$.

\begin{cor}\label{c: P_infty}
$P_\infty$ is a smoothing operator.
\end{cor}

\begin{proof}
This follows from Theorem~\ref{t: leafwise Hodge decomposition for d_FF} and Proposition~\ref{p: P} since $P_\infty=P_u\Pi_\FF$.
\end{proof}

\begin{prop}\label{p: P_u to P_infty}
$P_u\to P_\infty$ in $L(H^{-\infty}(M;\Lambda\FF),H^\infty(M;\Lambda\FF))$ as $u\uparrow\infty$.
\end{prop}

\begin{proof}
By Theorem~\ref{t: leafwise Hodge decomposition}, $e^{-u\Delta_\FF}-\Pi_\FF\to0$ in $\End(H^\infty(M;\Lambda\FF))$ as $u\uparrow\infty$. Therefore this convergence also holds in $\End(H^{-\infty}(M;\Lambda\FF))$, taking dual spaces and dual operators. Hence
\[
P_u-P_\infty=P_1(e^{-(u-1)\Delta_\FF}-\Pi_\FF)\to0
\]
in $L(H^{-\infty}(M;\Lambda\FF),H^\infty(M;\Lambda\FF))$ as $u\uparrow\infty$.
\end{proof}

\begin{cor}\label{c: K_P_u to K_P_infty}
$K_{P_u}\to K_{P_\infty}$ in $\Cinftyub(M^2;\Lambda\FF\boxtimes(\Lambda\FF^*\otimes\Omega M))$ as $u\uparrow\infty$.
\end{cor}

\begin{proof}
This follows from Propositions~\ref{p: Schwartz kernel with bd geometry} and~\ref{p: P_u to P_infty}.
\end{proof}

From now on, consider only the case where $A=f\in\Cinftyc(\R)$, obtaining the smoothing operator
\begin{equation}\label{P with f}
P=\int_\R\phi^{t*}\,f(t)\,dt\,\psi(D_\FF)\;,
\end{equation}
as well as its versions, $P_u$ if $\psi_u$ is used, and $P_\infty$ if $\Pi_\FF$ is used. The proof of \cite[Proposition~6.1]{AlvKordy2008a} clearly extends to the open manifold case, showing the following improvement of~\eqref{|P|_m,m'}.

\begin{prop}\label{p: |P|_m,m'}
For any compact $I\subset\R$ containing $\supp f$, and for all $m,m'\in\N_0$, there are some $C,C'>0$ and $N\in\N_0$, depending on $m$, $m'$ and $I$, such that
\[
\|P\|_{m,m'} \le C'\,\|f\|_{C^N,I}\|\psi\|_{\AA,C,N}\;.
\]
\end{prop}

Here, $\|\cdot\|_{C^N,I}$ is the semi-norm on $C^N(\R)$ defined by
\[
\|f\|_{C^N,I}=\max\{\,|f^{(m)}(x)|\mid x\in I,\ m=0,\dots,N\,\}\;.
\]

\section{Description of some Schwartz kernels}\label{s: Schwartz kernels}

Here, we will keep the setting of Section~\ref{s: smoothing operators}. The transverse vector field $\overline Z$ defines the structure of a transversely complete $\R$-Lie foliation on $\FF$ (Section~\ref{ss: Riem folns}). The corresponding Fedida's description of $\FF$ is given by a regular covering map $\pi:\widetilde M\to M$, a holonomy homomorphism $h:\Gamma:=\Aut(\pi)\to\R$, and the developing map $D:\widetilde M\to\R$ (Section~\ref{ss: Riem folns}). The lift of the bundle-like metric $g$ to $\widetilde M$ is a bundle-like metric $\tilde g$ of $\widetilde\FF=\pi^*\FF$, and let $\tilde\phi:\widetilde M\times\R\to\widetilde M$ and $\widetilde Z\in\fX_{\text{\rm ub}}(\widetilde M,\widetilde\FF)$ be the lifts of $\phi$ and $Z$. Then $\tilde g$ and $\widetilde Z$ are $\Gamma$-invariant, and $\tilde\phi$ is $\Gamma$-equivariant. Moreover $\widetilde Z$ is $D$-projectable, and we can assume that $D_*\widetilde Z=\partial_x\in\fX(\R)$, where $x$ denotes the standard global coordinate of $\R$. Thus $\phi$ induces via $D$ the flow $\bar\phi$ on $\R$ defined by $\bar\phi^t(x)=t+x$. Considering the equivalence between the holonomy pseudogroup and the pseudogroup generated by the action of $\Hol\FF$ on $\R$ by translations, this $\bar\phi$ corresponds to the equivariant local flow $\bar\phi$ induced by $\phi$ on the holonomy pseudogroup. Since $\bar\phi^t$ preserves every $\Hol\FF$-orbit in $\R$ if and only if $t\in\Hol\FF$, it follows that $\phi^t$ preserves every leaf of $\FF$ if and only if $t\in\Hol\FF$.

For any $\psi\in\AA$ and $f\in\Cinftyc(\R)$, we have the smoothing operator $P$ given by~\eqref{P with f}, and a similar smoothing operator $\widetilde P$ defined with $\tilde\phi$ and $\widetilde\FF$ instead of $\phi$ and $\FF$. We are going to describe their smoothing kernels under some assumptions.

Let $\fG=\Hol(M,\FF)$ and $\widetilde\fG=\Hol(\widetilde M,\widetilde\FF)$, whose source and range maps are denoted by $\bfs,\bfr:\fG\to M$ and $\tilde\bfs,\tilde\bfr:\widetilde\fG\to\widetilde M$ (Section~\ref{ss: holonomy groupoid}). Since the leaves of $\FF$ and $\widetilde\FF$ have trivial holonomy groups, the smooth immersions $(\bfr,\bfs):\fG\to M^2$ and $(\tilde\bfr,\tilde\bfs):\widetilde\fG\to\widetilde M^2$ are injective, with images $\RR_\FF$ and $\RR_{\widetilde\FF}$. Via these injections, the restriction $\pi\times\pi:\RR_{\widetilde\FF}\to\RR_\FF$ corresponds to the Lie groupoid homomorphism $\pi_\fG:=\Hol(\pi):\widetilde\fG\to\fG$ (Section~\ref{ss: fol maps}), which is a covering map with $\Aut(\pi_\fG)\equiv\Gamma$. In fact, since $\widetilde\FF$ is defined by the fiber bundle $D$, we get that $\RR_{\widetilde\FF}$ is a regular submanifold of $\widetilde M^2$, and $(\tilde\bfr,\tilde\bfs):\widetilde\fG\to\RR_{\widetilde\FF}$ is a diffeomorphism. We may write $\fG\equiv\RR_\FF$ and $\widetilde\fG\equiv\RR_{\widetilde\FF}$.

Consider the $C^\infty$ vector bundles, $S=\bfr^*\Lambda\FF\otimes\bfs^*(\Lambda\FF\otimes\Omega\FF)$ over $\fG$ and $\widetilde S={\tilde\bfr^*\Lambda\widetilde\FF}\otimes{\tilde\bfs^*(\Lambda\widetilde\FF\otimes\Omega\widetilde\FF)}$ over $\widetilde\fG$. Note that $\widetilde S\equiv\pi_\fG^*S$, and any $k\in C^\infty(\fG;S)$ lifts via $\pi_\fG$ to a section $\tilde k\in C^\infty(\widetilde{\fG};\widetilde S)$. Since $\pi$ restricts to diffeomorphisms of the leaves of $\widetilde\FF$ to the leaves of $\FF$, it follows that $\tilde k\in C^\infty_{\text{\rm p}}(\widetilde{\fG};\widetilde S)$ if and only if $k\in C^\infty_{\text{\rm p}}(\fG;S)$.

For any $\psi\in\RR$, the collection of Schwartz kernels $k_L:=K_{\psi(D_L)}$, for all leaves $L$ of $\FF$, defines a section $k=k_\psi$ of $S$ called {\em leafwise Schwartz kernel\/}, which {\em a priori\/} may not be continuous. This also applies to the operators $\psi(D_{\widetilde L})$ on the leaves $\widetilde L$ of $\widetilde\FF$, obtaining the leafwise Schwartz kernel $\tilde k=\tilde k_\psi$, which is a possibly discontinuous section of $\widetilde S$.

\begin{prop}\label{p: k in C^infty_p(fG;S)}
If $\hat\psi\in\Cinftyc(\R)$, then $k_\psi\in C^\infty_{\text{\rm p}}(\fG;S)$, and the global action of $k_\psi$ on $\Cinftyc(M;\Lambda\FF)$ {\rm(}Section~\ref{ss: global action}\/{\rm)} agrees with the restriction of the operator $\psi(D_\FF)$ on $H^\infty(M;\Lambda\FF)$ defined by Proposition~\ref{p: functional calculus} and~\eqref{Lambda M}.
\end{prop}

\begin{proof}
This follows with the arguments of \cite[Theorem~2.1]{Roe1987}, using $C^\infty_{\text{\rm p}}(\fG;S)$ instead of $\Cinftyc(\fG;S)$ when $M$ is not compact.
\end{proof}

\begin{rem}\label{r: k in C^infty_p(fG;S)}
In Proposition~\ref{p: k in C^infty_p(fG;S)}, more precisely, if $\supp\hat\psi\subset[-R,R]$ for some $R>0$, then $\supp k_\psi\subset\overline{\Pen}_\FF(\fG^{(0)},R)$ by~\eqref{supp hat psi subset[-R,R] => supp K_psi(D_z) subset ...}. Hence $\supp\psi(D_\FF)\alpha\subset\overline{\Pen}_\FF(\supp\alpha,R)$ for all $\alpha\in H^{-\infty}(M;\Lambda\FF)$ by Remark~\ref{r: supp K_A subset r-penumbra <=> supp Au subset r-penumbra for all u}.
\end{rem}

Suppose for a while that $\hat\psi\in\Cinftyc(\R)$. Then Proposition~\ref{p: k in C^infty_p(fG;S)} also applies to $\widetilde\FF$, obtaining that $\tilde k\in C^\infty_{\text{\rm p}}(\widetilde\fG;\widetilde S)$, and the global action of $\tilde k$ on $\Cinftyc(\widetilde M;\Lambda\widetilde\FF)$ is the restriction of the operator $\psi(D_{\widetilde\FF})$ on $H^\infty(\widetilde M;\Lambda\widetilde\FF)$ defined by Proposition~\ref{p: functional calculus} and~\eqref{Lambda M}. Indeed, since $\pi$ restricts to a diffeomorphism between the leaves of $\widetilde\FF$ and the leaves of $\FF$, we get that $\tilde k$ is the lift of $k$, and therefore the diagram
\begin{equation}\label{psi(D widetilde FF)}
\begin{CD}
\Cinftyc(\widetilde M;\Lambda\widetilde\FF) & @>{\psi(D_{\widetilde\FF})}>> & \Cinftyc(\widetilde M;\Lambda\widetilde\FF)\\
@V{\pi_*}VV & & @VV{\pi_*}V\\
\Cinftyc(M;\Lambda\FF) &@>{\psi(D_{{\FF}})}>> & \Cinftyc(M;\Lambda\FF)
\end{CD}
\end{equation}
is commutative, where
\begin{equation}\label{pi_*}
\pi_*\tilde\alpha\equiv\sum_{\gamma\in\Gamma}T_\gamma^*\tilde\alpha
\end{equation}
for all $\tilde\alpha\in\Cinftyc(\widetilde M;\Lambda\widetilde\FF)$, using the notation $T_\gamma$ for the action of every $\gamma\in\Gamma$ on $\widetilde M$. Locally, the series of~\eqref{pi_*} only has a finite number of nonzero terms. The same expression also defines $\pi_*:C^{-\infty}_{\text{\rm c}}(\widetilde M;\Lambda\widetilde\FF)\to C^{-\infty}_{\text{\rm c}}(M;\Lambda\FF)$.

Take some $R>0$ such that $\supp\hat\psi\subset[-R,R]$. By Remark~\ref{r: k in C^infty_p(fG;S)} and since $\phi$ is of $\R$-local bounded geometry, there is some $R'>0$ such that $\supp P\alpha\subset\overline{\Pen}(\supp\alpha,R')$ for all $\alpha\in H^{-\infty}(M;\Lambda\FF)$. Thus $P$ defines a continuous homomorphism $C^{-\infty}_{\text{\rm c}}(M;\Lambda\FF)\to\Cinftyc(M;\Lambda\FF)$. Similarly, $\widetilde P$ defines a continuous homomorphism $C^{-\infty}_{\text{\rm c}}(\widetilde M;\Lambda\widetilde\FF)\to\Cinftyc(\widetilde M;\Lambda\widetilde\FF)$. Moreover the commutativity of~\eqref{psi(D widetilde FF)} yields the commutativity of the diagram
\begin{equation}\label{widetilde P pi_* = pi_* P}
\begin{CD}
C^{-\infty}_{\text{\rm c}}(\widetilde M;\Lambda\widetilde\FF) & @>{\widetilde P}>> 
& \Cinftyc(\widetilde M;\Lambda\widetilde\FF)\phantom{\;.}\\
@V{\pi_*}VV & & @VV{\pi_*}V\\
C^{-\infty}_{\text{\rm c}}(M;\Lambda\FF) &@>P>> & \Cinftyc(M;\Lambda\FF)\;.
\end{CD}
\end{equation}

Let $\widetilde\Lambda=D^*dx\equiv dx$, which is a transverse invariant volume form of $\widetilde\FF$ defining the same transverse orientation as $\overline{\widetilde Z}$. Since $\widetilde\Lambda$ is $\Gamma$-invariant by the $h$-equivariance of $D$, it defines a transverse volume form $\Lambda$ of $\FF$, which defines the same transverse orientation as $\overline Z$. These $\widetilde\Lambda$ and $\Lambda$ define transverse invariant densities $|\widetilde\Lambda|$ and $|\Lambda|$ of $\widetilde\FF$ and $\FF$. 

\begin{prop}\label{p: Schwartz kernel}
Let $\psi\in\AA$ and $\tilde p,\tilde q\in\widetilde M$ over $p,q\in M$. Then, writing $t_{\tilde p,\tilde q}=D(\tilde q)-D(\tilde p)$ and using the identity $\widetilde S_{(\tilde p,\tilde q)}\equiv S_{(p,q)}$, we have\footnote{The leafwise part of the density of $K_P(\cdot,q)$ is given by the density of $\tilde k(\cdot,\tilde q)$.}
\[
K_P(p,q)
\equiv\sum_{\gamma\in\Gamma}
T_\gamma^*\,\tilde\phi^{t_{\tilde p,\tilde q}-h(\gamma)*}
\tilde k\big(T_\gamma\tilde\phi^{t_{\tilde p,\tilde q}-h(\gamma)}(\tilde p),\tilde q\big)\,f(t_{\tilde p,\tilde q})\,|\Lambda|(q)\;,
\]
defining a convergent series in $\Cinftyub(M^2;S)$.
\end{prop}

\begin{proof}
We can assume that $\hat\psi\in\Cinftyc(\R)$ by Propositions~\ref{p: |P|_m,m'} and~\ref{p: Schwartz kernel with bd geometry}, and because $\Cinftyc(\R)$ is dense in $\AA$. Then, by Proposition~\ref{p: k in C^infty_p(fG;S)}, for all $\tilde\alpha\in\Cinftyc(\widetilde M;\Lambda\widetilde\FF)$,
\begin{align*}
(\widetilde P\tilde\alpha)(\tilde p)&=\int_\R\big(\tilde\phi^{t*}\,\psi(D_{\widetilde\FF})\tilde\alpha\big)(\tilde p)\,f(t)\,dt\\
&=\int_\R\tilde\phi^{t*}\big(\psi(D_{\widetilde\FF})\tilde\alpha\big)\big(\tilde\phi^t(\tilde p)\big)\,f(t)\,dt\\
&=\int_{t\in\R}\int_{\tilde q\in D^{-1}(t)}\tilde\phi^{t*}\tilde k(\tilde\phi^t(\tilde p),\tilde q)\tilde\alpha(\tilde q)\,f(t)\,dt\\
&=\int_{\tilde q\in\widetilde M}\tilde\phi^{t_{\tilde p,\tilde q}*}
\tilde k(\tilde\phi^{t_{\tilde p,\tilde q}}(\tilde p),\tilde q)\tilde\alpha(\tilde q)\,
f(t_{\tilde p,\tilde q})\,\big|\widetilde\Lambda\big|(\tilde q)\;,
\end{align*}
because $D\tilde\phi^t(\tilde p)=D(\tilde q)$ if and only if $t=t_{\tilde p,\tilde q}$. Therefore
\[
K_{\widetilde P}(\tilde{p},\tilde{q})
=\tilde\phi^{t_{\tilde p,\tilde q}*}\tilde k\big(\tilde\phi^{t_{\tilde p,\tilde q}}(\tilde p),\tilde q\big)\,f(t_{\tilde p,\tilde q})\,
\big|\widetilde\Lambda\big|(\tilde q)\;.
\]
On the other hand, by~\eqref{pi_*} and the commutativity of~\eqref{widetilde P pi_* = pi_* P},
\[
K_P(p,q)
\equiv\sum_{\gamma\in\Gamma}
T_\gamma^*K_{\widetilde P}(\gamma\cdot\tilde p,\tilde q)\;.
\]
Locally, this series only has a finite number of nonzero terms. This is the series of the statement because $D(\tilde q)-D(\gamma\cdot\tilde p)=D(\tilde q)-D(\tilde p)-h(\gamma)$ by the $h$-equivariance of $D$.
\end{proof}

\section{Extension to the leafwise Novikov differential complex}\label{s: leafwise Novikov diff complex}

Consider the notation of Sections~\ref{ss: Novikov diff complex} and~\ref{ss: differential forms}, and assume that $\theta\in \Cinftyub(M;\Lambda^{0,1})\equiv \Cinftyub(M;\Lambda^1\FF)$. Then, like in~\eqref{d = d_0,1 + d_1,0 + d_2,-1} and~\eqref{delta = delta_0,-1 + delta_-1,0 + delta_-2,1}, we get the decompositions into bi-homogeneous components,
\[
d_z=d_{z,0,1}+d_{1,0}+d_{2,-1}\;,\quad\delta_z=\delta_{z,0,-1}+\delta_{-1,0}+\delta_{-2,1}\;,
\]
where $d_{z,0,1}=d_{0,1}+z\,{\theta\wedge}$ and $\delta_{z,0,-1}=\delta_{0,-1}-\bar z\,{\theta\!\lrcorner}$, which are of bounded geometry by Corollary~\ref{c: d_i,j, ... are of bd geom} and because $\theta\in \Cinftyub(M;\Lambda^{0,1})$. Since $\theta$ is closed, we get $d_{i,j}\theta=0$ for all $i,j$. So, by~\eqref{d_0,1^2=...=0} and~\eqref{delta_0,-1^2=...=0},
\begin{gather}
d_{z,0,1}^2=d_{z,0,1}d_{1,0}+d_{1,0}d_{z,0,1}=0\;,\label{d_z,0,1^2=...=0}\\
\delta_{z,0,-1}^2=\delta_{z,0,-1}\delta_{-1,0}+\delta_{-1,0}\delta_{z,0,-1}=0\;.\label{delta_z,0,-1^2=...=0}
\end{gather}
Let
\[
D_{0,z}=d_{z,0,1}+\delta_{z,0,-1}\;,\quad\Delta_{0,z}=D_{0,z}^2=d_{z,0,1}\delta_{z,0,-1}+d_{z,0,1}\delta_{z,0,-1}\;.
\]
On the other hand, we can also consider the leafwise version of the Novikov differential complex, $d_{\FF,z}=d_\FF+z\,{\theta\wedge}$ on $C^\infty(M;\Lambda\FF)$, or on $C^\infty(M;\Lambda\FF\otimes\Lambda N\FF)$, as well as its formal adjoint $\delta_{\FF,z}=\delta_\FF-\bar z\,{\theta\!\lrcorner}$. They satisfy the obvious versions of~\eqref{d_0,1 equiv d_FF} and Lemma~\ref{l: g bundle-like <=> delta_0,1 equiv delta_FF}, yielding obvious versions of~\eqref{d_0,1(f_IJ dx''^I wedge dx'^J)} and~\eqref{delta_0,-1(f_IJ dx''^I wedge dx'^J)}. Furthermore, for any choice of an orientation of $\FF$ on a distinguished open set $U$, we have
\[
\delta_{\FF,z}=(-1)^{n''v+n''+1}\star_\FF d_{\FF,-\bar z}\,\star_\FF
\]
on $C^\infty(U;\Lambda^v\FF)$ by~\eqref{delta_z = (-1)^nr+n+1 star d_-bar z star}. So, using also~\eqref{d_z,0,1^2=...=0} and~\eqref{delta_z,0,-1^2=...=0}, we get the following version of~\eqref{D_perp D_0 + D_0 D_perp}:
\[
D_\perp D_{0,z}+D_{0,z}D_\perp=KD_{0,z}+D_{0,z}K\;.
\]
This yields straightforward generalizations of all results and proofs of Section~\ref{s: Leafwise Hodge} for the {\em leafwise Novikov operators\/}, $d_{z,0,1}$, $D_{0,z}$, $\Delta_{0,z}$, $d_{\FF,z}$, $D_{\FF,z}$ and $\Delta_{\FF,z}$,. Let $\Pi_{0,z}$ and $\Pi_{\FF,z}$ denote the corresponding versions of $\Pi_0$ and $\Pi_\FF$. The term {\em leafwise Witten operators\/} should be used if $\theta$ is leafwise exact.

Let $\phi:(M,\FF)\to(M,\FF)$ be a smooth foliated map, let $\widetilde M$ be a regular covering of $M$ so that the lift $\tilde\theta$ of $\theta$ is exact, and let $\widetilde\FF$ be the lift of $\FF$ to $\widetilde M$. Like in the case of the Novikov differential complex (Section~\ref{ss: Novikov diff complex}), using~\eqref{phi^*}, any lift $\tilde\phi:(\widetilde M,\widetilde\FF)\to(\widetilde M,\widetilde\FF)$ of $\phi$ determines an endomorphism $\phi^{t*}_z$ of the leafwise Novikov differential complex $d_{\FF,z}$ on $C^\infty(M;\Lambda\FF)$, or on $C^\infty(M;\Lambda\FF\otimes\Lambda N\FF)$, which can be called a {\em leafwise Novikov perturbation\/} of $\phi^*$. With this definition, there is an obvious generalization of~\eqref{phi^*_0,0 equiv phi^*}, using the bi-homogeneous component $\phi^*_{z,0,0}$ of $\phi^*_z$ on $C^\infty(M;\Lambda)$. For every foliated flow $\phi=\{\phi^t\}$ on $(M,\FF)$, using its unique lift to a foliated flow $\tilde\phi=\{\tilde\phi^t\}$ on $(\widetilde M,\widetilde\FF)$, we get a unique determination of $\phi^{t*}_z$ on $C^\infty(M;\Lambda\FF)$, or on $C^\infty(M;\Lambda\FF\otimes\Lambda N\FF)$, called {\em the\/} Novikov perturbation of $\phi^{t*}$. Then, in Sections~\ref{s: smoothing operators} and~\ref{s: Schwartz kernels}, the definitions of $P$, $P_u$, $P_\infty$, $k$, $\tilde k$, $k_u$ and $\tilde k_u$ can be extended by using $\phi^{t*}_z$ and $D_{\FF,z}$ instead of $\phi^{t*}$ and $D_\FF$. The subindex ``$z$'' may be added to their notation if needed. Moreover the results, proofs and observations of Sections~\ref{s: smoothing operators} and~\ref{s: Schwartz kernels} have straightforward generalizations to this setting, using the indicated extensions of the tools.

\bibliographystyle{amsplain}



\providecommand{\bysame}{\leavevmode\hbox to3em{\hrulefill}\thinspace}
\providecommand{\MR}{\relax\ifhmode\unskip\space\fi MR }
\providecommand{\MRhref}[2]{%
  \href{http://www.ams.org/mathscinet-getitem?mr=#1}{#2}
}
\providecommand{\href}[2]{#2}

\end{document}